\documentclass[preprint,12pt]{elsarticle}

\usepackage{amssymb}   
\usepackage{amsthm}    
\usepackage{amsmath}   
\usepackage{stmaryrd}  
\usepackage{titletoc}  
\usepackage{mathrsfs}  
\usepackage{cases}

\theoremstyle{plain}
\newtheorem{thm}{Theorem}[section]
\newtheorem{cor}[thm]{Corollary}
\newtheorem{lem}[thm]{Lemma}
\newtheorem{prop}[thm]{Proposition}

\theoremstyle{definition}
\newtheorem{defn}{Definition}[section]

\newtheorem{rem}{Remark}[section]
\newcommand{\R}{\mathbb{R}}
\newcommand{\E}{\mathbb{E}}

\newcommand{\bH}{\mathbb{H}}

\newcommand{\cD}{\mathcal{D}}

\newcommand{\cT}{\mathcal{T}}

\newcommand{\bfH}{\textbf{H}}

\newcommand{\eps}{\varepsilon}
\newcommand{\vf}{\varphi}
\newcommand{\la}{\langle}
\newcommand{\ra}{\rangle}

\newcommand{\ptl}{\partial}
\newcommand{\rrow}{\rightarrow}

\newcommand{\wt}{\widetilde}
\newcommand{\wh}{\widehat}

\newcommand{\supp}{\mathop{\textrm{supp}}}

\makeatletter
   
   \@addtoreset{equation}{section}
\makeatother

\journal{}

\begin{document}

\begin{frontmatter}



\title{Backward stochastic partial differential equations with quadratic
growth\tnoteref{label1}}

\tnotetext[label1]{Supported by NSFC Grant \#10325101, NNSF Grant \#11101140, by Basic Research
Program of China (973 Program) Grant \#2007CB814904,  and by the Science Foundation
for Ministry of Education of China Grant \#20090071110001.}

\author{Kai Du}
\ead{kdu@fudan.edu.cn}
\address{
School of Mathematical Sciences, Fudan University, Shanghai 200433, China.}

\author{Shaokuan Chen}
\ead{shaokuan.chen@utdallas.edu}
\address{School of Management, University of Texas at Dallas, Richardson, TX, 75080, USA.}

\begin{abstract}
This paper is concerned with the existence and uniqueness of weak solutions to
the Cauchy-Dirichlet problem of backward stochastic partial differential
equations (BSPDEs) with nonhomogeneous terms of quadratic growth in both the
gradient of the first unknown and the second unknown. As an example, we
consider a non-Markovian stochastic optimal control problem with cost
functional formulated by a quadratic BSDE, where the corresponding value
function satisfies the above quadratic BSPDE.
\end{abstract}


\end{frontmatter}

\section{Introduction}

Denote by $\mathcal{T}$ the fixed time duration $[0,T]$. Let $(\Omega,
\mathscr{F},\{\mathscr{F}_t\}_{t\in \mathcal{T}},\mathbb{P})$ be a complete
filtered probability space on which a $d_0$-dimensional standard Wiener process
$W_t =(W^1_t,\dots,W^{d^0}_t)^{\prime }$ is defined such that $\{
\mathscr{F}_t\}_{t\in \mathcal{T}}$ is the natural filtration generated by $W $
and augmented by all the $\mathbb{P}$-null sets in $\mathscr{F}$. We denote by
$\mathscr{P}$ the predictable $\sigma$-algebra associated with $\{
\mathscr{F}_t\}_{t\in \mathcal{T}}$.

In this paper we consider the Cauchy-Dirichlet problem for the following
parabolic quadratic backward stochastic partial differential equation (BSPDE in
short)
\begin{equation}  \label{c5:main}
du = -\big[(a^{ij}u_{x^j} + \sigma^{ik}q^k)_{x^i} + f(t,x,u,u_x,q)\big]\,dt +
q^k dW^k_t,
\end{equation}
with the terminal-boundary condition
\begin{equation}  \label{c5:con}
\left\{
\begin{array}{ll}
u(t,x)=0,~~~~~ & t\in \mathcal{T},\ x\in\partial \mathcal{D}, \\
u(T,x)=\varphi(x),~~~~~ & x\in \mathcal{D}.
\end{array}
\right.
\end{equation}
Here $\mathcal{D}$ is a simply connected and bounded region in the Euclidean
space $\mathbb{R}^d$ and we use the convention that repeated indices imply
summation. By the terminology ``super-parabolic" (resp., ``degenerate") we
mean the condition that there exist positive constants $\kappa$ and $K $
such that
\begin{equation}  \label{c5:sparab}
\kappa I_d + (\sigma^{ik})(\sigma^{jk})^* \leq 2(a^{ij}) \leq K I_d.
\end{equation}
\begin{equation}  \label{c5:degrat}
\big( \text{resp.},~~ 2a^{ij}-(\sigma^{ik})(\sigma^{jk})^*\geq0.\big)
\end{equation}
And ``quadratic" means that
\begin{center}
$|f(t,x,v,p,r)| \leq \lambda_0(t,x) + \lambda_1 |v| + \gamma(|v|)(|p|^2 +
|r|^2),$
\end{center}
for some positive constant $\lambda_1$, bounded predictable field $\lambda_0$,
and increasing function~$\gamma(\cdot):\mathbb{R}_+\rightarrow\mathbb{R}_+$.
We refer to $(f,\varphi)$ as the parameters of BSPDE \eqref{c5:main}-\eqref{c5:con}.

BSPDEs are generalized backward stochastic differential equations (BSDEs in
short) with values in function spaces. Linear BSDEs were initiated by Bismut \cite{Bis1} as
the adjoint equations in 1973 when he studied the stochastic maximum principle of stochastic optimal control problems. In 1990,
Pardoux and Peng \cite{PaPe90} introduced the general nonlinear BSDEs with
Lipschitz continuous generators. In the last two decades, extensive research on such kind of
equations has indicated that BSDEs can serve as a
powerful tool in many fields such as mathematical finance, stochastic control,
and partial differential equations (PDEs in short). See among others
\cite{KaPeQu}, \cite{PaPe2}, \cite{PaPe3}, \cite{Peng92b}. Since this paper is
inspired by the study of quadratic BSDEs, a kind of BSDEs with generators of
quadratic growth in the martingale term, we mainly introduce the development in
this direction. The motivation of studying quadratic BSDEs derives from the
feedback representation of the optimal control in the setting of linear quadratic stochastic
control problem where the related backward stochastic Riccati equation (BSRE in
short) turns out to be a quadratic BSDE. Bismut \cite{Bis78} first considered
BSREs in a special case where the generator depends on the second unknown in a
linear way. Later Peng \cite {Peng92a} applied Bellman's principle of quasi-linearization to deal with relatively general BSREs. In 2000, Kobylanski \cite{Koby00} developed a quite useful technique, the idea of which is from the Cole-Hopf
transformation in PDE theory, to overcome the difficulty from the quadratic growth of the generator
in the martingale term and obtained the existence result in one-dimensional case. Numerous literatures later
were devoted to solving the challenging problem concerning the existence and
uniqueness of solutions to BSREs in multidimensional case, among which we refer \cite{ChYo01}, \cite
{KoTa02}, \cite{KoTa03}, \cite{KoZh00}, \cite{Lim04}, \cite{LiZh02} and the
references therein to show the theoretical developments. Until 2003, Tang \cite{Tang03} gave a complete solution
to this long standing problem by a new constructive method. In addition,
general BSDEs with quadratic features are also applied to describe the value
functions and corresponding optimal trading strategies in utility maximization
problems (see e.g. \cite{HuImMu05}, \cite{Seki06}) and appear naturally in the
study of the BSDEs on manifolds (see e.g. \cite{Blac05}, \cite{Blac06}). For
the recent theoretical progress on the existence and uniqueness of solutions to
quadratic BSDEs, one can refer to \cite{BrHu06}, \cite{BrHu08}.

As the infinite dimensional counterparts of BSDEs, BSPDEs also arise from
stochastic optimal control theory. For instance, they serve as the adjoint
equations in the formulation of the stochastic maximum principle for controlled
stochastic differential equations (SDEs) with partially observed information
(see e.g. \cite{Bens83}, \cite{Tang98}) or controlled stochastic parabolic
partial differential equations (see e.g. \cite{NaNi90}, \cite {Zhou93}). Value
functions of the stochastic optimization problem of controlled non-Markovian
SDEs, according to Bellman's optimal principle and It\^{o}-Wentzell's formula,
have been shown to satisfy the so-called stochastic Hamilton-Jacobi-Bellman
(HJB) equations, a class of fully nonlinear BSPDEs (see e.g. \cite{Peng92a}).
As for their important applications to issues from financial models with random
parameters, we refer to \cite{EnKa09} and \cite{MaYo97}.

The theory of BSPDEs is more rich since such equations have features of both
BSDEs and PDEs. The theory on existence, uniqueness and regularity of solutions
to Cauchy problem of BSPDEs is fairly complete. See \cite{DuQiTa11},
\cite{Zhou92a}, \cite{DuMe10} for non-degenerate BSPDEs, and \cite {DuTaZh11},
\cite{HuMaYo02}, \cite{MaYo99}, \cite{Tang05} for the more difficult degenerate
case. However, discussions on Cauchy-Dirichlet problem are relatively less, and
one can refer to \cite{DuTang11} and \cite{Tess96}. Methods mainly applied to
handle BSPDEs include: techniques of semigroup of operators in the case of
BSPDEs with deterministic coefficients, adjoint arguments closely related to
the theory of forward SPDEs, probabilistic representation methods depending on
the theory of forward-backward stochastic differential equations (FBSDEs), and
PDE's techniques, such as frozen coefficient method and continuation method.
The last two methods are proved to be powerful to handle degenerate BSPDEs. To
the best of our knowledge, in the literature the nonhomogeneous term of a BSPDE
has at most linear growth in the second unknown. As a generalization of the
BSDE considered by Kobylanski \cite{Koby00} to infinite dimensional case, in
this paper we first explore the Cauchy-Dirichlet problem of a BSPDE with a
nonhomogeneous term that has quadratic growth with respect to both the gradient
of the first unknown and the second unknown. A change of variables scheme is
implemented to establish the existence and uniqueness of weak solutions to
non-degenerate BSPDEs with the above quadratic nonhomogeneous terms. We also
demonstrate its application in a non-Markovian stochastic optimal control
problems. As indicated in Remark \ref{rem4.1}, our approaches and results can
be easily extended to the case of the whole space $\R^d$, that is, the Cauchy
problem.

The rest of this paper is organized as follows. In section 2 we introduce some
notations and preliminary results. Section 3 and section 4 are devoted to the
existence and uniqueness of solutions to the Cauchy-Dirichlet problems of
quadratic BSPDEs, respectively. In section 5, an example of quadratic BSPDEs is
demonstrated in the context of a stochastic optimal control problem with cost
functional formulated by a quadratic BSDE.

\section{Notations and preliminaries}

For a given Banach space $\mathcal{B}$ and a constant $p\in[1,\infty]$, we denote
by $L^{p}_{\mathscr{P}}(\Omega\times\mathcal{T};\mathcal{B})$ the space of all
$\mathcal{B}$-valued predictable processes $X:\Omega\times
\mathcal{T}\rightarrow\mathcal{B}$ such that $\mathbb{E}\int_0^T\|X_t\|_
\mathcal{B}^pdt<\infty$. We also denote by $C(\mathcal{T};\mathcal{B})$ the space of
all $\mathcal{B}$-valued continuous adapted processes $X:\Omega\times
\mathcal{T}\rightarrow\mathcal{B}$ such that $\mathbb{E}\sup_{t\in\mathcal{T}
}\|X_t\|^2_\mathcal{B}<\infty$ and by $L^p(E)$ the space of all real
valued measurable functions $f$ defined on a measure space $(E,\mathcal{E}
,\mu)$ such that $\int_E|f|^pd\mu<\infty$. For simplicity we denote
\begin{equation*}
\mathbb{L}^{p}:=L^{p}_{\mathscr{P}}(\Omega\times\mathcal{T};L^p(\mathcal{D})).
\end{equation*}

For a vector $q\in\mathbb{R}^d$, $q^i$ means its $i$-th component, $
i=1,2,\cdots,d$. For a function $u$ defined on $\mathbb{R}^d$, $u_{x^i}$ or $
D_iu$ means the derivative of $u$ with respect to $x^i$. $u_x$ or $Du$, stands
for the the gradient of $u$, and $D^2u$ stands for the Hessian of $u$.

For a integer $m$, we simply denote by $H^{m}(\mathcal{D})$ and $H^{m}_0(
\mathcal{D})$ the Sobolev spaces $W^{m,2}(\mathcal{D})$ and $W^{m,2}_0(
\mathcal{D})$, respectively, with inner product $(\cdot,\cdot)_m$. With the
above notations, we simply denote
\begin{align*}
&\mathbb{H}^{m}(\mathcal{D}) :=L^{2}_{\mathscr{P}}(\Omega\times\mathcal{T};
H^{m}(\mathcal{D})),~~m=-1,0,1,2,\dots, \\
&\mathbb{H}^{n}_{0}(\mathcal{D}) :=L^{2}_{\mathscr{P}}(\Omega\times\mathcal{T
}; H^{n}_{0}(\mathcal{D})),~~n=1,2,3,\dots, \\
&\mathbb{H}^{m}(\mathcal{D};\mathbb{R}^{d_0}) :=\big(\mathbb{H}^{m}(\mathcal{
D})\big)^{d_0}.
\end{align*}
And we denote by $C_0^\infty(\mathcal{D})$ the space of infinitely differential
real functions with compact support defined on $\mathcal{D}$.

We first introduce the notation of weak solution to the BSPDE \eqref{c5:main} .

\begin{defn}
\label{c3:ws.def}  A pair of random fields $(u,q)\in \mathbb{H}^{1}_{0}(
\mathcal{D})\times  \mathbb{H}^{0}(\mathcal{D};\mathbb{R}^{d_0})$ is said to be
a weak solution  to BSPDE \eqref{c5:main} if for every $\eta \in
C_{0}^{\infty}(\mathcal{D})$,
\begin{eqnarray}  \label{chap3:weaksol}
\begin{split}
&\int_{\mathcal{D}}u(t,x)\eta(x)dx \\
=&\int_{\mathcal{D}}\varphi(x)\eta(x)dx+\int_{t}^{T}\int_{\mathcal{D}} \Big[-
\big( a^{ij}(s,x)u_{x^j}(s,x)+\sigma^{ik}(s,x)q^{k}(s,x)\big)\eta_{x^i}(x) \\
&+f(s,x,u(s,x),u_x(s,x),q(s,x))\eta(x)\Big]\,dxds -\int_{t}^{T}\int_{
\mathcal{D}}q^{k}(s,x)\eta(x)dxd\,W^{k}_{s},~~~~d\mathbb{P}\times dt\text{
-a.e.}.
\end{split}
\end{eqnarray}
\end{defn}

We present a generalized It\^{o}'s formula and a comparison principle for weak
solutions to BSPDEs, the proof of which one can refer to \cite{DeMaSt09} or
\cite{QiuTang11}.

\begin{lem}
\label{c3:2-lem1}  Suppose $f^0\in\mathbb{L}^{1}(\mathcal{D})$, $f^i, q^k\in
\mathbb{H}^{0}(\mathcal{D}),~i=1,\dots,d,~k=1,\dots,d_0$, and  $u\in \mathbb{
H}^{1}_0(\mathcal{D})\cap C(\mathcal{T};L^2(\mathcal{D}))$. If for any  ~$
\eta\in C^\infty_0(\mathcal{D})$,
\begin{equation*}
(u(t),\eta)_0 = (u(T),\eta)_0 + \int_t^T \Big[(f^0(s),\eta)_0 -
(f^i(s),\eta_{x^i})_0 \Big]\,ds - \int_t^T (q^k(s),\eta)_0\,dW^k_s,~~~\forall
t\in\mathcal{T},~~~a.s.,
\end{equation*}
it holds that for every $\psi$  such that $\psi^{\prime }$ and $\psi^{\prime
\prime }$ are bounded and $\psi^{\prime }(0)=0$,
\begin{eqnarray}  \label{c3:2-2-002}
\begin{split}
\int_{\mathcal{D}} & \psi(u(t,x))\,dx - \int_{\mathcal{D}} \psi(u(T,x))\,dx
\\
=~& \int_t^T\!\!\int_{\mathcal{D}}\Big[\psi^{\prime 0 }- \psi^{\prime \prime
}(u)u_{x^i} f^i -\frac{1}{2}\psi^{\prime \prime 2}\Big](s,x)\,dxds \\
&- \int_t^T\!\!\int_{\mathcal{D}}\psi^{\prime k}(s,x)\,dxdW^k_s,~~~\forall
t\in\mathcal{T},~~~a.s..
\end{split}
\end{eqnarray}
\end{lem}

\begin{lem}
\label{lem2}  Let $(u_1,q_1)$, $(u_2,q_2)\in\mathbb{H}^{1}_{0}(\mathcal{D}
)\times  \mathbb{H}^{0}(\mathcal{D};\mathbb{R}^{d_0})$ be weak solutions to
BSPDEs  with parameters $(f_1,\varphi_1)$ and $(f_2,\varphi_2)$, respectively.
Assume

(i) For any $(v,p,r)\in\mathbb{R}\times\mathbb{R}^d\times\mathbb{R}^{d_0}$, $
f_i(\cdot,\cdot,\cdot,v,p,r)$  is $\mathscr{P}\times\mathscr{B}(\mathcal{D})$
measurable, $i=1,2$. Moreover, there exists  a constant $L>0$ such that for any
$(v_1,p_1,r_1)$, $(v_2,p_2,r_2)\in\mathbb{R}\times\mathbb{R}^d\times
\mathbb{R}^{d_0}$,
\begin{equation*}
\begin{split}
&|f_i(t,x,v_1,p_1,r_1)-f_i(t,x,v_2,p_2,r_2)| \\
\leq~&L(|v_1-v_2|+|p_1-p_2|+|r_1-r_2|),\ \forall (\omega,t,x)\in\Omega\times
\mathcal{T}\times\mathcal{D},\ i=1,2.
\end{split}
\end{equation*}

(ii) $f_i(\cdot,0,0,0)\in\mathbb{H}^{0}(\mathcal{D})$, $i=1,2$.

(iii) $\varphi_i:\Omega\times\mathcal{D}\rightarrow\mathbb{R}$ is $
\mathscr{F}_T\times\mathscr{B}(\mathcal{D})$  measurable, and $\varphi_i\in
L^2(\Omega\times\mathcal{D})$, $i=1,2$.\newline If $\varphi_1\leq\varphi_2$ and
$f_1\leq f_2$, we have $u_1\leq u_2$.
\end{lem}

Using a similar procedure in Proposition \ref{c5:2-prp1}, we have

\begin{cor}
\label{c3:5-002.cor}  Let the parameters $(f,\varphi)$ of BSPDE
\eqref{c5:main}-\eqref{c5:con} satisfy  the assumptions in Lemma \ref{lem2} and
let $(u,q)$ be a weak solution to  BSPDE \eqref{c5:main}-\eqref{c5:con} with
parameters $(f,\varphi)$.  Suppose $\zeta:\mathcal{T} \rightarrow[0,\infty)$
satisfies the ODE $\dot{\zeta}(t)=-g(t,\zeta(t))$. Then, if $f(\omega, t, x,
\zeta(t),0,0)\leq g(t,\zeta(t))$, we have
\begin{equation*}
u(t,x)\leq \zeta(t),\ d\mathbb{P}\times dx\ a.e.,\ \forall t\in\mathcal{T}.
\end{equation*}
\end{cor}

Finally, we give a simple but useful result, which will be used frequently in
the subsequent argument.

\begin{lem}
\label{c5:2-2-002}  Let $\mu_{0}=\frac{\kappa}{1+2K}$. Then for any vectors $
p\in\mathbb{R}^d$ and $r\in\mathbb{R}^{d_0}$, it holds that
\begin{equation}  \label{c5:2-2-0021}
2 a^{ij} p^i p^j + 2\sigma^{ik} p^i r^k + |r|^2 \geq \mu_{0} (|p|^2 + |r|^2).
\end{equation}
\end{lem}

\section{The existence of solutions}

Throughout this paper we always assume that coefficients $a^{ij}=a^{ji}$ and
$\sigma^{ik}$ are $\mathscr{P}\times\mathscr{B}(\mathcal{D})$ measurable and
bounded functions, $i,j=1,\dots,d,~k=1,\dots,d_0$. As for coefficients $f$ and
$\varphi$, we assume in this section \medskip

(\textbf{H}1) (i) For every $(v,p,r)\in\mathbb{R}\times\mathbb{R}^d\times
\mathbb{R}^{d_0}$, $f(\cdot,\cdot,\cdot,v,p,r)$ is $\mathscr{P}\times
\mathscr{B}(\mathcal{D})$ measurable. And for every $(\omega,t,x)$, $f$ is
continuous with respect to $(v,p,r)$.

(ii) There exist a positive function $\lambda_0\in\mathbb{L}^{\infty}\cap
\mathbb{L}^{2}$, a positive constant $\lambda_1$ and a increasing function $
\gamma(\cdot):\mathbb{R}_+\rightarrow\mathbb{R}_+$ such that for every $
(\omega,t,x,v,p,r)$,

\begin{center}
$|f(t,x,v,p,r)| \leq \lambda_0(t,x) + \lambda_1 |v| + \gamma(|v|)(|p|^2 +
|r|^2).$
\end{center}

\medskip

(\textbf{H}2)~~$\varphi$: $\Omega\times\mathcal{D}\rightarrow \mathbb{R}$ is
$\mathscr{F}_T\times\mathscr{B}(\mathcal{D})$ measurable and $\varphi\in
L^\infty(\Omega\times\mathcal{D})\cap L^2(\Omega\times\mathcal{D})$.\medskip

The main theorem of this section is

\begin{thm}
\label{c5:2-thm1}  Suppose
\eqref{c5:sparab},~(\textbf{H}1)~and~(\textbf{H}2)~hold. Then there exists a
weak solution $(u,q)  \in \mathbb{H}^1_0(
\mathcal{D})\times\mathbb{H}^0(\mathcal{D};\mathbb{R}^{d_0})$ to BSPDE~
\eqref{c5:main}-\eqref{c5:con},  and $u\in L^2(\Omega;C(\mathcal{T};L^2(
\mathcal{D})))\cap\mathbb{L}^{\infty}$.
\end{thm}

\subsection{Boundedness and convergence}

To prove Theorem \ref{c5:2-thm1}, we need to establish a prior estimates. To
the end of this subsection, we first strengthen the condition (ii) in
(\textbf{H}1) to the case
\begin{equation}  \label{c5:2-2-001}
|f(t,x,v,p,r)| \leq \lambda_0(t,x) + \lambda_1 |v| + \lambda \mu_{0} (|p|^2 +
|r|^2),
\end{equation}
where $\lambda$ is a positive constant.

\begin{prop}
\label{c5:2-prp1}
Let~\eqref{c5:sparab},~\eqref{c5:2-2-001}~and~(\textbf{H}2)~be satisfied.
Suppose  $(u,q)\in \mathbb{H}^1_0(\mathcal{D})\times\mathbb{
H}^0(\mathcal{D};\mathbb{R}^{d_0})$~is a weak solution to BSPDE  ~
\eqref{c5:main}-\eqref{c5:con} and in addition $u\in C(\mathcal{T};L^2(
\mathcal{D}))\cap\mathbb{L}^{\infty}$. Then,
\begin{equation}  \label{c5:2-2-004}
\|u(t,\cdot)\|_{L^\infty(\Omega\times\mathcal{D})} \leq \frac{\|\lambda_0\|_{
\mathbb{L}^{\infty}}}{\lambda_1}\big(e^{\lambda_1 (T-t)}-1\big)+
e^{\lambda_1(T-t)}\|\varphi\|_{L^\infty(\Omega\times\mathcal{D})},\ \forall
t\in\mathcal{T}.
\end{equation}
Moreover, there exists a constant $C_1$ depending only on~$\|u\|_{\mathbb{L}
^{\infty}}$,  $\|\varphi(x)\|_{L^2(\Omega\times\mathcal{D})}$, $
\|\lambda_0\|_{\mathbb{L}^{2}}$, $\mu_{0}$, $\lambda$,  $\lambda_1$ and $T$,
such that
\begin{equation}  \label{c5:2-2-005}
\|u_x\|_{\mathbb{H}^0(\mathcal{D})}^2 + \|q\|_{\mathbb{H}^0(\mathcal{D})}^2
\leq C_1.
\end{equation}
\end{prop}

The next result shows that the existence of solution to BSPDE
\eqref{c5:main}-\eqref{c5:con} can be obtained by an approximation scheme.

\begin{prop}
\label{c5:prop.01}  Suppose that a sequence of functions $(f^n)_{n\geq1}$ and
$f$ satisfy (\textbf{H}1)  and that a sequence of functions~$
(\varphi^n)_{n\geq1}$ and $\varphi$ satisfy (\textbf{H}2). Furthermore, we
assume

(a) For every~$(\omega,t,x)$, the sequence $(f^n)_n$ converges to $f$ on  $
\mathbb{R}\times\mathbb{R}^d\times\mathbb{R}^{d_0}$ locally uniformly and the
sequence $(\varphi^n)_n$ converges to $\varphi$  in~$L^{\infty}(\Omega
\times\mathcal{D})$ as $n\rightarrow\infty$.

(b)~There exist a positive constant $\lambda$ and a function  $\lambda_2\in
\mathbb{L}^{\infty}\cap\mathbb{L}^{2}$ such that

$|f^n(t,x,v,p,r)|\leq \lambda_2(t,x) + \lambda \mu_{0}  (|p|^2+|r|^2)$, $
\forall~n\in \mathbb{N},~\forall~(\omega,t,x,v,p,r)$.

(c) For every $n\in\mathbb{N}$, BSPDE with parameters $(f^n,\varphi^n)$ has a
weak solution

$(u^n,q^n)  \in \mathbb{H}^1_0(\mathcal{D})\times\mathbb{H}^0(\mathcal{D};
\mathbb{R}^{d_0}),~~~~  u^n\in L^2(\Omega;C(\mathcal{T};L^2(\mathcal{D}
)))\cap\mathbb{L}^{\infty},$

and $(u^n)_n$ is a monotone sequence. Moreover, there exists a positive
constant $M$ such that  $\|u^n\|_{\mathbb{L}^\infty}\leq M$ for all $n\in
\mathbb{N}$.

Then, BSPDE \eqref{c5:main}-\eqref{c5:con} has a weak solution $(u,q)  \in
\mathbb{H}^1_0(\mathcal{D})\times\mathbb{H}^0(\mathcal{D};\mathbb{R}^{d_0})$
which satisfies
\begin{equation*}
\begin{split}
\displaystyle\lim_{n\rightarrow\infty}u^n&=u\ \text{uniformly on}\
\Omega\times\mathcal{T}\times\mathcal{D}, \\
\displaystyle\lim_{n\rightarrow\infty}q^n&=q\ \text{in}\ \mathbb{H}^0(
\mathcal{D};\mathbb{R}^{d_0}),
\end{split}
\end{equation*}
and moreover $u\in L^2(\Omega;C(\mathcal{T};L^2(\mathcal{D})))\cap\mathbb{L}
^{\infty}$.
\end{prop}

The proofs of the above two Propositions are both technical and lengthy and
thus are arranged in the appendix section.

\subsection{Change of variables}

This section is devoted to the change of variables between two weak solution.
To be more precise, we justify that the exponential change of variables of a
weak solution to some BSPDE is also a weak solution to another corresponding
BSPDE. This technique is crucial in the proof of Theorem \ref{c5:2-thm1} in the
next section.

We consider the following equaiton
\begin{equation}  \label{c5:2-2-eq1}
\left\{
\begin{array}{l}
dv = -\big[(a^{ij}v_{x^j} + \sigma^{ik}r^k)_{x^i} + F\big]\,dt + r^k dW^k_t,
\\
v|_{\mathcal{T}\times\partial\mathcal{D}} = 0,~~~~v(T) = e^{\lambda \varphi}-1,
\end{array}
\right.
\end{equation}
where $F$ is a function defined on $\Omega\times\mathcal{T}\times\mathcal{D}$ .
Set
\begin{equation}  \label{c5:2-2-007}
u = \frac{1}{\lambda}\ln (v+1),~~~~q = \frac{r}{\lambda(v+1)}.
\end{equation}
Applying It\^o's formula formally to $u$, we obtain $(u,q)$ satisfies the BSPDE
\begin{equation}  \label{c5:2-2-eq2}
\left\{
\begin{array}{l}
du = -\big[(a^{ij}u_{x^j} + \sigma^{ik}q^k)_{x^i} + f(t,x,u,u_x,q)\big]\,dt
+ q^k dW^k_t, \\
u|_{\mathcal{T}\times\partial\mathcal{D}} = 0,~~~~u(T) = \varphi,
\end{array}
\right.
\end{equation}
where
\begin{equation}  \label{c5:2-2-008}
f(t,x,u,u_x,q) = \lambda^{-1}e^{-\lambda u}F(t,x) + \lambda
(a^{ij}u_{x^i}u_{x^j} + \sigma^{ik}u_{x^i}q^k + \frac{1}{2}|q|^2).
\end{equation}

\begin{lem}
\label{c5:2-lem2} Suppose $F$~and~$\varphi$ are both bounded functions. Let $
(v,r)$ be a weak solution to BSPDE \eqref{c5:2-2-eq1} and satisfy ~$
0<\gamma\leq v+1 \leq \Gamma<\infty$, where $\gamma$, $\Gamma$ are constants.
Then the pair of random fields $(u,q)$ defined by \eqref{c5:2-2-007} is a weak
solution to BSPDE \eqref{c5:2-2-eq2}.
\end{lem}

\begin{proof}
For any given test function $\eta\in C^{\infty}_0(\cD)$, set $K :=\supp
(\eta)$, $\eps_0 = \textrm{dist}(K,\ptl\cD)$.

We further choose a nonnegative function $\zeta\in C^{\infty}_0(\cD)$ such
that:
$$\supp (\zeta) \subset \{|x| \leq 1\},~~\int_{\R^d}\zeta(x)dx = 1
,$$ and define
\begin{equation*}
  \begin{split}
  \zeta^{\eps}(x)&~=
\eps^{-d}\zeta(x/\eps),\ \ \forall\eps \in (0,\eps_0),\\
v^{\eps}(x)&~= \zeta^{\eps} * v (x),\ \forall x\in K,\\
r^{\eps}(x)&~= \zeta^{\eps}*r (x),\ \forall x\in K.
\end{split}
\end{equation*}

Since $(v,r)$ is a weak solution to equation \eqref{c5:2-2-eq1}, we know from
the definition of $(v^{\eps},r^{\eps})$ that the pair $(v^{\eps},r^{\eps})$
satisfies£º
\begin{eqnarray*}
  \begin{split}
    v^{\eps}(t,x) = ~&v^{\eps}(T,x) + \int_{t}^T \Big\{D_i\big[\zeta^{\eps}*
    (a^{ij}D_j v + \sigma^{ik}r^k)\big] + \zeta^{\eps}*F\Big\}(s,x)ds \\&-
    \int_t^T r^{\eps,k} (s,x) dW^k_s,\ \forall(t,x)\in\cT\times K.
  \end{split}
\end{eqnarray*}

Setting
\begin{equation*}
  u^{\eps} = \frac{1}{\lambda}\ln(v^{\eps}+1),~~~~
  q^{\eps} = \frac{r^{\eps}}{\lambda(v^{\eps}+1)},
\end{equation*}
then we have $u^{\eps}_x = \frac{v^{\eps}_x}{\lambda(v^{\eps}+1)}$. Applying
It\^{o}'s formula to $u^{\eps}$, we get
\begin{eqnarray*}
  \begin{split}
    &u^{\eps}(t,x) - u^{\eps}(T,x)\\
    &= \int_t^T \frac{1}{\lambda(v^{\eps}+1)}\cdot
    \Big\{D_i\big[\zeta^{\eps}*
    (a^{ij}D_j v + \sigma^{ik}r^k)\big] + \zeta^{\eps}*F\Big\}(s,x)ds\\
    &~~~~ +\int_t^T \frac{\lambda}{2}|q^{\eps}(t,x)|^2 dt
    - \int_t^T q^{\eps}(s,x) dW_s,\ \ \forall(t,x)\in\cT\times K.
  \end{split}
\end{eqnarray*}
Multiplying $\eta$ on both sides of the above equality and integrating over
$K$, applying Fubini's theorem and the fact that $K =\supp (\eta)$, we have
\begin{eqnarray}
  \begin{split}
    &\int_{K} u^{\eps}(t,x)\eta(x)\,dx - \int_{K}u^{\eps}(T,x)\eta(x)\,dx\\
    &= \int_t^T \!\!\int_{K}\frac{1}{\lambda(v^{\eps}+1)}\cdot
    \Big\{D_i\big[\zeta^{\eps}*
    (a^{ij}D_j v + \sigma^{ik}r^k)\big] + \zeta^{\eps}*F\Big\}(s,x)\eta(x)\,dxds\\
    &~~~~ +\frac{\lambda}{2} \int_t^T\!\!\int_{K} |q^{\eps}(s,x)|^2\eta(x) \,dxds
    - \int_t^T\!\!\int_{K} q^{\eps}(s,x)\eta(x) \,dxdW_s.
  \end{split}
\end{eqnarray}
Green's formula yields
\begin{eqnarray}\label{c5:2-2-010}
  \begin{split}
    &\int_{K} u^{\eps}(t,x)\eta(x)\,dx - \int_{K}u^{\eps}(T,x)\eta(x)\,dx\\
    &= - \int_t^T \!\!\int_{K}
    \big[\zeta^{\eps}*
    (a^{ij}v_{x^j} + \sigma^{ik}r^k)\big]
    \frac{\ptl}{\ptl x^i}\Big[\frac{\eta}{\lambda(v^{\eps}+1)}\Big](s,x) \,dxds\\
    &~~~~ +\int_t^T\!\!\int_{K}\Big\{
    \frac{1}{\lambda(v^{\eps}+1)}(\zeta^{\eps}*F) + \frac{\lambda}{2}
    |q^{\eps}|^2\Big\}(s,x)\eta(x)\, ds
    - \int_t^T\!\!\int_{K} q^{\eps}(s,x)\eta(x) \,dxdW_s.
  \end{split}
\end{eqnarray}

In what follows we will take limits as $\eps\rrow 0$ on both sides of
\eqref{c5:2-2-010}. First, for every $h\in L^{p}(\R^d)$ with $p\in[1,\infty))$,
$\zeta^{\eps}*h$~converges strongly to $h$ in $L^{p}(\R^d)$ as $\eps\rightarrow
0$. Moreover, we know that

(i) It is easy to verify that $\lambda^{-1}\ln\gamma\leq u \leq
\lambda^{-1}\ln\Gamma$;

(i) The fact that $(v,r)$ is a weak solution to equation \eqref{c5:2-2-eq1}
indicates that $v\in \bH^1(\cD), r\in\bH^0(\cD)$, which implies $a^{ij}v_{x^j}
+ \sigma^{ik}r^k\in \bH^0(\cD)$;

(iii) Since $v$ is bounded, $q\in\bH^0(\cD)$.\\
Then, letting $\eps\rrow 0$ in the equality \eqref{c5:2-2-010}, we have
\begin{eqnarray}
  \begin{split}
    &\int_{K} u(t,x)\eta(x)\,dx - \int_{K}u(T,x)\eta(x)\,dx\\
=&~- \int_t^T \!\!\int_{K}
    (a^{ij}v_{x^j} + \sigma^{ik}r^k)
    \frac{\ptl}{\ptl x^i}\Big[\frac{\eta}{\lambda(v+1)}\Big](s,x) \,dxds\\
    &~+ \int_t^T\!\!\int_{K}\Big\{
    \frac{F}{\lambda(v+1)} + \frac{\lambda}{2}
    |q|^2\Big\}(s,x)\eta(x) ds
    - \int_t^T\!\!\int_{K} q(s,x)\eta(x) \,dxdW_s.
  \end{split}
\end{eqnarray}
Substituting $e^{\lambda u}-1$ and $\lambda e^{\lambda u} q$ for $v$ and $r$
respectively, we have
\begin{eqnarray}\label{c5:2-2-009}
  \begin{split}
    &\int_{\cD} u(t,x)\eta(x)\,dx - \int_{\cD}u(T,x)\eta(x)\,dx\\
=&~- \int_t^T \!\!\int_{\cD}
    (a^{ij}u_{x^j} + \sigma^{ik}q^k)(s,x)
    \eta_{x^i}(x) \,dxds\\
    &~+ \int_t^T\!\!\int_{\cD}f(s,x,u,u_x,q)\eta(x) dxds
    - \int_t^T\!\!\int_{\cD} q(s,x)\eta(x) \,dxdW_s,
  \end{split}
\end{eqnarray}
where $f(s,x,u,u_x,q)$ is given in \eqref{c5:2-2-008}.

We can deduce from the arbitrariness of $\eta$ in \eqref{c5:2-2-009} that
$(u,q)$ is a weak solution to equation \eqref{c5:2-2-eq2}.
\end{proof}

\subsection{Proof of Theorem \protect\ref{c5:2-thm1}}

We are now in a position to prove Theorem \ref{c5:2-thm1}. We first assume
\eqref{c5:2-2-001} holds, i.e.,

\begin{center}
$|f(t,x,v,p,r)| \leq \lambda_0(t,x) + \lambda_1 |v| + \lambda \mu_{0} (|p|^2 +
|r|^2).$
\end{center}

Denote $M = \frac{\|\lambda_0\|_{\mathbb{L}^{\infty}}}{\lambda_1} (e^{\lambda_1
T}-1)+  e^{\lambda_1T}\|\varphi\|_{L^\infty(\Omega\times \mathcal{D})} $. By
Proposition \ref{c5:2-prp1}, if $(u,q)$ is a weak solution to BSPDE
\eqref{c5:main}-\eqref{c5:con} and $u$ is bounded, we have $u(\cdot)\leq M$.

Set
\begin{equation*}
v = e^{2\lambda u} - 1,~~~~r = 2\lambda e^{2\lambda u} q.
\end{equation*}
Then $(v,r)$ formally satisfies
\begin{equation}
\left\{
\begin{array}{l}
dv = -\big[(a^{ij}v_{x^j} + \sigma^{ik}r^k)_{x^i} + F(t,x,v,v_x,r)\big]\,dt
+ r^k dW^k_t, \\
v|_{\mathcal{T}\times\partial\mathcal{D}} = 0,~~~~v(T) = e^{2\lambda
\varphi}-1,
\end{array}
\right.
\end{equation}
where
\begin{eqnarray*}
\begin{split}
F(t,x,v,p,r) = ~& 2\lambda(v+1)f\bigg(t,x,\frac{1}{2\lambda}\ln(v+1), \frac{p
}{2\lambda(v+1)},\frac{r}{2\lambda(v+1)}\bigg) \\
&-\frac{1}{2(v+1)}(2 a^{ij}p^ip^j + 2\sigma^{ik}p^ir^k + |r|^2).
\end{split}
\end{eqnarray*}

Take a function $\psi\in C^\infty$ such that
\begin{equation*}
\psi(z) = \left\{
\begin{array}{ll}
1,~~~~ & z\in [e^{-2\lambda M},e^{2\lambda M}], \\
0,~~~~ & z\notin [e^{-2\lambda (M+1)},e^{2\lambda (M+1)}],
\end{array}
\right.
\end{equation*}
and denote $\widetilde{F}(t,x,v,p,r) = \psi(v+1)F(t,x,v,p,r)$. From
\eqref{c5:2-2-001} and Lemma \ref{c5:2-2-002}, we know
\begin{equation}  \label{c5:2-2-011}
\begin{split}
&-\psi(v+1)\Big[2\lambda(\|\lambda_0\|_{\mathbb{L}^{\infty}} + \lambda_1 +
\lambda_1 M) (v+1) + \frac{C_{K,\mu_0}}{v+1}(|p|^2 + |r|^2)\Big] \\
\leq&~ \widetilde{F}(t,x,v,p,r) \leq~ 2\lambda(\|\lambda_0\|_{\mathbb{L}
^{\infty}} + \lambda_1 + \lambda_1 M) \psi(v+1) (v+1),
\end{split}
\end{equation}
where $C_{K,\mu_0}$ is a constant depending on $K$ and $\mu_0$ .

Using the same method as \cite[pp.~572~]{Koby00}, we can construct a sequence
of functions $\{F^{n}(t,x,v,p,r):n\geq 1\}$ such that

(a) For every $n$ and any $(\omega,t,x)$, $F^{n}(t,x,v,p,r)$ is uniformly
Lipschitz continuous with respect to $(v,p,r)$.

(b) The sequence $(F^{n})_n$ is decreasing, and for almost every $
(\omega,t,x)$, $F^{n}(t,x,v,p,r)$ locally uniformly converges to $\widetilde{
F}(t,x,v,p,r)$ on $\mathbb{R}\times\mathbb{R}^d\times\mathbb{R}^{d_0}$.
Additionally,
\begin{equation*}
\begin{split}
\widetilde{F}(t,x,v,p,r)&~\leq F^{n}(t,x,v,p,r) \\
&~\leq 2\lambda(\|\lambda_0\|_{\mathbb{L}^{\infty}} + \lambda_1 + \lambda_1 M)
\psi(v+1) (v+1) + 2^{-n}.
\end{split}
\end{equation*}

By Lemma 2.3 in \cite{DuTang11}, BSPDE
\begin{equation*}
\left\{
\begin{array}{l}
dv^{n} = -\big[(a^{ij}v^{n}_{x^j} + \sigma^{ik}r^{n,k})_{x^i} +
F^{n}(t,x,v^{n},v^{n}_x,r^{n})\big]\,dt + r^{n,k} dW^k_t, \\
v^{n}|_{\mathcal{T}\times\partial\mathcal{D}} = 0,~~~~v^{n}(T) = e^{2\lambda
\varphi}-1,
\end{array}
\right.
\end{equation*}
has a unique weak solution $(v^{n},r^{n})\in \mathbb{H}^{1}_0(\mathcal{D}
)\times\mathbb{H}^0(\mathcal{D})$. Moreover, on account of Lemma \ref{lem2}, we
know that for every $n\in \mathbb{N}$, $v^{n+1}\leq v^{n}$. On the other hand,
applying Corollary \ref{c3:5-002.cor} and meanwhile noticing inequality
\eqref{c5:2-2-011}, we have

\begin{center}
$e^{-2\lambda (M+1)}-1 \leq v^n \leq e^{2\lambda (M+1)}.$
\end{center}

Setting
\begin{equation*}
u^n ~=~\displaystyle\frac{\ln (v^n+1)}{2\lambda},~~~~~~  q^n~=~\displaystyle
\frac{r^n}{2\lambda (v^n+1)},
\end{equation*}
we can deduce from Lemma \ref{c5:2-lem2} that $(u^n,q^n)$ is a weak solution to
the equation
\begin{equation*}
\left\{
\begin{array}{l}
du^{n} = -\big[(a^{ij}u^{n}_{x^j} + \sigma^{ik}q^{n,k})_{x^i} +
f^{n}(t,x,u^{n},u^{n}_x,q^{n})\big]\,dt + q^{n,k} dW^k_t, \\
u^{n}|_{\mathcal{T}\times\partial\mathcal{D}} = 0,~~~~u^{n}(T) = \varphi,
\end{array}
\right.
\end{equation*}
where
\begin{align*}
f^{n}(t,x,v,p,r) ~=~ & \frac{1}{2\lambda e^{2\lambda v}} F^n(t,x,e^{2\lambda
v}-1,2\lambda e^{2\lambda v} p,2\lambda e^{2\lambda v} r) \\
&+\lambda (2 a^{ij}p^ip^j + 2\sigma^{ik}p^i r^k + |r|^2).
\end{align*}
In view of the properties (a) and (b) of $F^n$, it is easy to verify that $
f^n$ satisfies the conditions in Proposition \ref{c5:prop.01} and the
corresponding limit is
\begin{align*}
&\widetilde{f}(t,x,v,p,r) \\
&= ~ \frac{1}{2\lambda e^{2\lambda v}} \widetilde{F}(t,x,e^{2\lambda
v}-1,2\lambda e^{2\lambda v} p,2\lambda e^{2\lambda v} r) +\lambda (2
a^{ij}p^ip^j + 2\sigma^{ik}p^i r^k + |r|^2) \\
&=~ \psi(e^{2\lambda v})f(t,x,v,p,r) + [1-\psi(e^{2\lambda v})]\lambda (2
a^{ij}p^ip^j + 2\sigma^{ik}p^i r^k + |r|^2).
\end{align*}
Therefore, from Proposition \ref{c5:prop.01} we know the following equation
\begin{equation*}
\left\{
\begin{array}{l}
d\widetilde{u} = -\big[(a^{ij}\widetilde{u}_{x^j} + \sigma^{ik}\widetilde{q}
^{k})_{x^i} + \widetilde{f}(t,x,\widetilde{u},\widetilde{u}_x,\widetilde{q})
\big]\,dt + \widetilde{q}^{k} dW^k_t, \\
\widetilde{u}|_{\mathcal{T}\times\partial\mathcal{D}} = 0,~~~~\widetilde{u} (T)
= \varphi,
\end{array}
\right.
\end{equation*}
has at least a weak solution $(\widetilde{u},\widetilde{q})\in \mathbb{H}
^1_0(\mathcal{D})\cap\mathbb{L}^{\infty}\times\mathbb{H}^0(\mathcal{D};
\mathbb{R}^{d_0})$. Furthermore, according to Proposition \ref{c5:2-prp1} we
have $|\widetilde{u}(\cdot)|\leq M$. We notice that when $|v|\leq M$,

\begin{center}
$\widetilde{f}(t,x,v,p,r) = {f}(t,x,v,p,r).$
\end{center}
Therefore $(\widetilde{u},\widetilde{q})$ is also a weak solution to BSPDE
\eqref{c5:main}-\eqref{c5:con}.

Finally we prove the existence of solution to BSPDE \eqref{c5:main}-
\eqref{c5:con} in the general case, i.e., condition \eqref{c5:2-2-001} is
replaced by (\textbf{H}1),

\begin{center}
$|f(t,x,v,p,r)| \leq \lambda_0(t,x) + \lambda_1 |v| + \gamma(|v|)(|p|^2 +
|r|^2).$
\end{center}

Since the estimate \eqref{c5:2-2-004} in Proposition \ref{c5:2-prp1} is
independent of $\lambda$, we can use the truncation technique to complete the
proof.

Denote the sets
\begin{align*}
&E^+:= \{(\omega,t,x,v,p,r)~|~ f > \lambda_0 + \lambda_1 |v| +
\gamma(M)(|p|^2 + |r|^2)\}, \\
&E^-:= \{(\omega,t,x,v,p,r)~| - f > \lambda_0 + \lambda_1 |v| + \gamma(M)(|p|^2
+ |r|^2)\},
\end{align*}
and recall $M = \frac{\|\lambda_0\|_{\mathbb{L}^{\infty}}}{\lambda_1}
(e^{\lambda_1 T}-1)+  e^{\lambda_1T}\|\varphi\|_{L^\infty(\Omega\times
\mathcal{D})}$. Let
\begin{equation*}
\widehat{f}(\omega,t,x,v,p,r)=\left\{
\begin{array}{ll}
\lambda_0 + \lambda_1 |v| + \gamma(M)(|p|^2 + |r|^2),~~ &
(\omega,t,x,v,p,r)\in E^+, \\
{f}(\omega,t,x,v,p,r), & (\omega,t,x,v,p,r)\notin E^+\cup E^-, \\
-\lambda_0 -\lambda_1 |v| - \gamma(M)(|p|^2 + |r|^2),~~ & (\omega,t,x,v,p,r)\in
E^-.
\end{array}
\right.
\end{equation*}
Obviously $\widehat{f}$ satisfies condition \eqref{c5:2-2-001}. It follows from
the previous arguments that BSPDE
\begin{equation*}
\left\{
\begin{array}{l}
d\widehat{u} = -\big[(a^{ij}\widehat{u}_{x^j} + \sigma^{ik}\widehat{q}
^{k})_{x^i} + \widehat{f}(t,x,\widehat{u},\widehat{u}_x,\widehat{q})\big]
\,dt + \widehat{q}^{k} dW^k_t, \\
\widehat{u}|_{\mathcal{T}\times\partial\mathcal{D}} = 0,~~~~\widehat{u}(T) =
\varphi,
\end{array}
\right.
\end{equation*}
has at least a weak solution $(\widehat{u},\widehat{q}) \in \mathbb{H}^1_0(
\mathcal{D})\times\mathbb{H}^0(\mathcal{D};\mathbb{R}^{d_0})$~and~$|\widehat{
u}|\leq M$. However, when~$|v|\leq M$,

\begin{center}
$\widehat{f}(t,x,v,p,r) = {f}(t,x,v,p,r).$
\end{center}

Therefore, $(\widehat{u},\widehat{q})$ is also a weak solution to BSPDE
\eqref{c5:main}-\eqref{c5:con}. The proof is complete.

\section{The uniqueness of solutions}

Let $M>0$ be a fixed constant. For simplicity, we denote by $z = (p,r)\in
\mathbb{R}^{d+d_0}$ for vectors $p\in\mathbb{R}^d$ and $r\in\mathbb{R}^{d_0}$ .
We assume $f(\omega,t,x,u,z) = f(\omega,t,x,u,p,r)$ satisfies

(\textbf{H}3) there exist functions $l(\cdot)\in L^1(\mathcal{T}\times
\mathcal{D})\cap L^\infty(\mathcal{T}\times\mathcal{D}),\ k(\cdot)\in L^2(
\mathcal{T})$ and a positive constant $\Lambda$ such that for any $(t,x)\in
\mathcal{T}\times\mathcal{D}$, $u\in[-M,M]$, and $z\in \mathbb{R}^{d+d_0}$,
\begin{gather*}
|f(t,x,u,z)|\leq l(t) + \Lambda |z|^2,~~a.s., \\
|f_z(t,x,u,z)| \leq k(t) + \Lambda |z|,~~a.s..
\end{gather*}
Moreover, for any $\varepsilon > 0$, there exists $l_\varepsilon(\cdot)\in
L^1(\mathcal{T})$ such that for every $(t,x)\in\mathcal{T}\times\mathcal{D}$ ,
$u\in\mathbb{R}$, $z\in \mathbb{R}^{d+d_0}$,
\begin{equation*}
|f_u(t,x,u,z)|  \leq l_{\varepsilon}(t) + \varepsilon|z|^2,~~a.s..
\end{equation*}

The main theorem of this section concerns the uniqueness of solutions to BSPDE
\eqref{c5:main}-\eqref{c5:con}.

\begin{thm}
\label{c5:3-thm1}  Let condition (\textbf{H}3) be satisfied. Suppose $
(u^1,q^1)$ and $(u^2,q^2)$ are both weak solutions to BSPDE \eqref{c5:main}-
\eqref{c5:con} and $|u^1|,\ |u^2|\leq M$. Then $u^1 = u^2$.
\end{thm}

\begin{proof}
\emph{The first step}. We first prove this theorem under a more stringent
condition. Assume

(\bfH 4) There exist a positive constant $a$ and a function $b(\cdot)\in
L^1(\cT)$ such that for every $(t,x)\in\cT\times\cD$, $u\in[-M,M]$, $z\in
\R^{d+d_0}$, \begin{center}$
  f_u(t,x,u,z)
  + a |f_z(t,x,u,z)|^2
  \leq b(t),~~a.s..
$\end{center}

Denote $\wh{u} = u^1-u^2$ and $\wh{q} = q^1-q^2$. Set $\wh{u}^+ =
\max(0,\wh{u})$. Applying It\^o's formula (Lemma \ref{c3:2-lem1}), we have for
any $m\geq 2$ and $m\in\mathbb{N}$,
\begin{eqnarray}\label{c5:3-001}
  \begin{split}
    &\int_{\cD}[\wh{u}^+(t,x)]^{2m}\,dx \\
    &+ m(2m-1)\int_t^T\!\!\int_{\cD}
    (\wh{u}^+)^{2(m-1)}(2 a^{ij}\wh{u}_{x^i}\wh{u}_{x^j}
    + 2\sigma^{ik}\wh{u}_{x^i}\wh{q} + |\wh{q}|^2)(s,x)\,dxds\\
=&~ 2m\int_t^T\!\!\int_{\cD}(\wh{u}^+)^{2m-1}\wh{f}(s,x)\,dxds
    -2m\int_t^T\!\!\int_{\cD}(\wh{u}^+)^{2m-1}\wh{q}(s,x)\,dxdW_s.
  \end{split}
\end{eqnarray}
Here
\begin{eqnarray*}
  \begin{split}
    \wh{f}(s,x) = & f(s,x,u^1,u^1_x,q^1) - f(s,x,u^2,u^2_x,q^2)\\
    = & \bigg(\int_0^1 f_u(\Xi)d\lambda\bigg)\wh{u}
    + \bigg(\int_0^1 f_z(\Xi)d\lambda\bigg)(\wh{u}_x,\wh{q})',
  \end{split}
\end{eqnarray*}
where
$$(\Xi):= (s,x,\lambda u^1 + (1-\lambda)u^2,\lambda u^1_x + (1-\lambda)u^2_x,
\lambda q^1 + (1-\lambda)q^2).$$ Cauchy-Schwarz's inequality yields
\begin{eqnarray*}
  \begin{split}
    (\wh{u}^+)^{2m-1}\wh{f} \leq
    \bigg(\int_0^1 (f_u + a|f_z|^2)(\Xi)d\lambda\bigg)(\wh{u}^+)^{2m}
    + \frac{1}{4a}(\wh{u}^+)^{2(m-1)}(|\wh{u}_x|^2 + |\wh{q}|^2).
  \end{split}
\end{eqnarray*}
Noticing the inequality \eqref{c5:2-2-0021} and assumption (\bfH4), we can
deduce from \eqref{c5:3-001} that
\begin{eqnarray*}
  \begin{split}
    &\int_{\cD}[\wh{u}^+(t,x)]^{2m}\,dx
    + m\Big[\mu_{0}(2m-1)-\frac{1}{2a}\Big]\int_t^T\!\!\int_{\cD}
    (\wh{u}^+)^{2(m-1)}(|\wh{u}_x|^2 + |\wh{q}|^2)(s,x)\,dxds\\
\leq&~ 2m\int_t^T\!\!\int_{\cD}b(s)(\wh{u}^+)^{2m}(s,x)\,dxds
    -2m\int_t^T\!\!\int_{\cD}(\wh{u}^+)^{2m-1}\wh{q}(s,x)\,dxdW_s.
  \end{split}
\end{eqnarray*}
Taking expectation on both sides of the above ineuality, we have
\begin{eqnarray*}
  \begin{split}
    &\E\int_{\cD}[\wh{u}^+(t,x)]^{2m}\,dx
    + m\Big[\mu_{0}(2m-1)-\frac{1}{2a}\Big]\E\int_t^T\!\!\int_{\cD}
    (\wh{u}^+)^{2(m-1)}(|\wh{u}_x|^2 + |\wh{q}|^2)(s,x)\,dxds\\
\leq&~ 2m\int_t^T\!\!\int_{\cD}b(s)\E(\wh{u}^+)^{2m}(s,x)\,dxds.
  \end{split}
\end{eqnarray*}
Choosing $m$ large enough such that $\mu_{0}(2m-1)-\frac{1}{2a}\geq 0$,
together with Gronwall's inequality, we know that
$\E\int_{\cD}[\wt{u}^+(t,x)]^{2m}\,dx = 0$,  for all $t\in\cT$. So $u^1\leq
u^2$.

In the same way we can prove $u^2\leq u^1$. Hence $u^1 = u^2$.

\emph{The second step}. We will search for an appropriate change of variables
to convert BSPDE \eqref{c5:main}-\eqref{c5:con} satisfying (\bfH3) to another
BSPDE satisfying condition (\bfH4). Let
$$\wt{u} = \phi^{-1}(u),~~~~\wt{q} = q/w(u),$$
where $\phi$ is a smooth and increasing function to be determined with the
condition $\phi(0)=0$ and $w(u) = \phi'(\wt{u}) = \phi'(\phi^{-1}(u))$.

Suppose $(u,q)$ is a weak solution to BSPDE \eqref{c5:main}-\eqref{c5:con}.
Analogous to the proof of Lemma \ref{c5:2-lem2}, it is easy to verify that
$(\wt{u},\wt{q})$ is a weak solution to the equation
\begin{equation}\label{c5:3-003}
  \left\{\begin{array}{l}
    d\wt{u} = -\big[(a^{ij}\wt{u}_{x^j} + \sigma^{ik}\wt{q}^k)_{x^i}
    + F(t,x,\wt{u},\wt{u}_x,\wt{q})\big]\,dt
    + \wt{q}^k dW^k_t,\\
    \wt{u}|_{\cT\times\ptl\cD} = 0,~~~~\wt{u}(T) = \phi^{-1}(\vf),
  \end{array}\right.
\end{equation}
where
\begin{eqnarray}\label{c5:3-002}
  \begin{split}
    F(t,x,\wt{u},\wt{u}_x,\wt{q}) =~ & \frac{1}{\phi'(\wt{u})}
    \bigg[ f(t,x,\phi(\wt{u}),\phi'(\wt{u})\wt{u}_x,\phi'(\wt{u})\wt{q})\\
    & +\frac{1}{2}\phi''(\wt{u})( 2a^{ij}\wt{u}_{x^i}\wt{u}_{x^j}
    + 2\sigma^{ik}\wt{u}_{x^i}\wt{q}^k + |\wt{q}|^2) \bigg].
  \end{split}
\end{eqnarray}
Therefore it is sufficient to prove the equation \eqref{c5:3-003} has a unique
bounded weak solution.

We still denote $z = (u_x,q)$ and $\wt{z} = (\wt{u}_x,\wt{q})$. Obviously $z =
\phi'(\wt{u})\wt{z}$. We denote by $\la A(t,x)\wt{z},\wt{z} \ra$ the positive
definite quadratic form
$$2a^{ij}\wt{u}_{x^i}\wt{u}_{x^j}
    + 2\sigma^{ik}\wt{u}_{x^i}\wt{q}^k + |\wt{q}|^2,$$
where $A$ is a function with value in the space of symmetric positive definite
matrices. We can deduce from Lemma \ref{c5:2-2-002} that
$$\la A(t,x)\wt{z},\wt{z} \ra   = \wt{z} A(t,x)\wt{z}' \geq \mu_{0} |\wt{z}|^2.$$
So \eqref{c5:3-002} can be rewritten as
\begin{equation*}
    F(t,x,\wt{u},\wt{z}) = \frac{1}{\phi'(\wt{u})}
    \bigg[ f(t,x,\phi(\wt{u}),\phi'(\wt{u})\wt{z})
    +\frac{1}{2}\phi''(\wt{u})\la A(t,x)\wt{z},\wt{z} \ra  \bigg].
\end{equation*}
By simple computation,
\begin{eqnarray*}
  \begin{split}
    F_{\wt{u}}(t,x,\wt{u},\wt{z}) ~=~&
    - \frac{w'}{w} f(t,x,u,z) + f_u(t,x,u,z)
    + \frac{w''}{2w}\la A(t,x)z,z \ra \\
    & + \frac{w'}{w}z f_z(t,x,u,z)\\
    =~& \frac{w''}{2w}\la A z,z \ra
    + \frac{w'}{w}(z f_z-f) + f_u,\\
    F_{\wt{z}}(t,x,\wt{u},\wt{z}) = & f_z(t,x,u,z) + \frac{w'}{w}z.
  \end{split}
\end{eqnarray*}
If we can choose an appropriate $\phi$ such that $w>0$, $w'>0$ and $w''<0$,
from (\bfH3) we have
\begin{eqnarray*}
  \begin{split}
    & \Big(F_{\wt{u}} + a
    |F_{\wt{z}}|^2\Big)(t,x,\wt{u},\wt{z}) \\
    ~= ~& \frac{w''}{2w}\la A z,z \ra
    + \frac{w'}{w}(z f_z-f) + f_u + a \bigg|f_z + \frac{w'}{w}z\bigg|^2\\
    \leq~ & \frac{\mu_{0}}{2}\frac{w''}{w}|z|^2
    + \frac{w'}{w}\Big[l(t)+k(t)|z|+2\Lambda|z|^2\Big]
    + l_{\eps}(t) + \eps |z|^2\\
    &+ a \bigg[k(t)
    + \bigg(\Lambda+\frac{w'}{w}\bigg)|z|\bigg]^2\\
    \leq~ & |z|^2\bigg[\frac{\mu_{0}}{2}\frac{w''}{w}
    +2\Lambda\frac{w'}{w}+\eps+  a\bigg(\Lambda+\frac{w'}{w}\bigg)^2
    \bigg]\\
    &+ |z|\bigg[ k(t)\frac{w'}{w} + 2ak(t)\bigg(\Lambda+\frac{w'}{w}\bigg)\bigg]
    + l(t)\frac{w'}{w}+l_{\eps}(t) + a[k(t)]^2\\
    \leq~ & |z|^2\bigg[\frac{\mu_{0}}{2}\frac{w''}{w}
    +2\Lambda\frac{w'}{w}+\bigg(\frac{w'}{w}\bigg)^2 + \eps+
    2a\bigg(\Lambda+\frac{w'}{w}\bigg)^2
    \bigg]\\
    & + l(t)\frac{w'}{w}+l_{\eps}(t) + (1+2a)[k(t)]^2.
  \end{split}
\end{eqnarray*}
Once we find a function $\phi$ such that besides $w(u)>0$, $w'(u)>0$ and
$w''(u)<0$ on $[-M,M]$,
$$\frac{\mu_{0}}{2}\frac{w''}{w}
    +2\Lambda\frac{w'}{w}+\bigg(\frac{w'}{w}\bigg)^2 \leq -\delta<0,$$
we can choose $a$ and $\eps$ small enough to assure that $F(t,x,\wt{u},\wt{z})$
satisfies condition (\bfH4). Then we will obtain the desired result.

Set
$$ u = \phi(\wt{u}) = \frac{1}{\beta}\ln
\frac{(B e^{\beta M}-1)e^{\beta B \widetilde{u}}+1}{B e^{\beta M}},$$ where
$B>1$ and $\beta>0$ are constants to be determined. Obviously $\phi$ is a
strictly increasing function and $\phi(0)= 0$. By computation we know that for
any $u\in[-M,M]$,
\begin{eqnarray*}
  w(u) & =& B - e^{-\beta(u + M)}~>~0~,\\
  w'(u)& =& \beta e^{-\beta(u + M)} ~>~0~,\\
  w''(u)& =& - \beta^2 e^{-\beta(u + M)} ~<~0~.
\end{eqnarray*}
Furthermoer,
\begin{eqnarray*}
  \begin{split}
    &\frac{\mu_{0}}{2}\frac{w''}{w}
    +2\Lambda\frac{w'}{w}+\bigg(\frac{w'}{w}\bigg)^2\\
    =~&-\frac{\beta e^{-\beta(u + M)}}{(B - e^{-\beta(u + M)})^2}
    \bigg[\bigg(\frac{\mu_{0}}{2}\beta
    - 2\Lambda\bigg) B + \bigg(2\Lambda B - \frac{\mu_{0}+2}{2}\beta\bigg)
    e^{-\beta(u+M)}\bigg].
  \end{split}
\end{eqnarray*}
We can choose appropriate $\beta$ and $B$ to assure the above equality
negative. The proof is complete.
\end{proof}

\begin{rem}\label{rem4.1}
In the case $\mathcal{D}=\mathbb{R}^d$, we claim that the conclusions
concerning the existence and uniqueness of solutions to BSPDEs in bounded
domains are still valid. Indeed, setting $g^i=f^i-u_{x^i}$, we know that $
(u,q)$ is a weak solution to the BSPDE
\begin{equation*}
\left\{
\begin{split}
&du~=~-\big(\Delta u+f^0+\sum_{i=1}^{d}g^i_{x^i}\big)dt+q^kdW^k_t, \\
&u(T,x)~=~\varphi(x).
\end{split}
\right.
\end{equation*}
Approximating the coefficients $f^0$, $g^i$, $i=1,\cdot,d$, and $\varphi$ by
sequences of functions in the space $C_0^\infty(\mathbb{R}^d)$, and applying
Corollary 3.4 in \cite{DuTang11}, we can prove that the It\^{o}'s formula in
Lemma \ref{c3:2-lem1} is still valid. Once the It\^{o}'s formula is
established, we can obtain the claim since in addition to the assumptions on
the boundedness of coefficients, we require their corresponding integrability
in appropriate spaces to avoid the item $\emph{meas}(\mathcal{D })$ appearing
in the estimates.
\end{rem}

\section{Application to non-Markovian stochastic control problems}

Analogous to \cite{Peng92a}, in this section we give an example, a stochastic
control problem with a recursive cost functional formulated by a quadratic
BSDE, to illustrate that the corresponding value function will formally satisfy
a kind of stochastic Hamilton-Jacobi-Bellman equations with quadratic growth.

The stochastic HJB equation that we concern has the form
\begin{equation}  \label{e1}
\left\{
\begin{split}
-du(t,x)=&\frac{1}{2}tr\big\{[
\sigma(t,x)\sigma^*(t,x)+\pi(t,x)\pi^*(t,x)]D^2_xu(t,x)\big\}dt \\
&+\inf_{v\in V}\big\{f\big(t,x,u(t,x),(\langle
D_xu(t,x),\sigma(t,x)\rangle+q(t,x), \\
&~\ \ \ \ \ \ \ \ \ \langle D_xu(t,x),\pi(t,x)\rangle),v\big)+\langle
b(t,x,v),D_xu(t,x)\rangle\big\}dt \\
&+\langle D_xq(t,x),\sigma(t,x)\rangle dt+q(t,x)dW_t, \\
u(T,x)=&~\phi(x).
\end{split}
\right.
\end{equation}
In what follows, we show its formal derivation from the context of a
non-Markovian stochastic control problem.

Suppose $(B_t)_{t\in\mathcal{T}}$ is another standard Wiener process which is
independent of $(W_t)_{t\in\mathcal{T}}$. Without loss of generality, we only
consider the case where $B$ and $W$ are both one-dimensional. Denote by
$\{\mathscr{F}^*_t\}_{t\in\mathcal{T}}$ is the natural filtration generated by
both $W$ and $B$ and augmented by all the $\mathbb{P}$-null sets in $
\mathscr{F}$. We also denote by $\mathscr{P}^*$ the predictable $\sigma$
-algebra associated with $\{\mathscr{F}^*_t\}_{t\in \mathcal{T}}$.

We introduce the admissible control set
\begin{equation*}
\begin{split}
\mathcal{V}_{t,T}:=\big\{v(\cdot)| &v(\cdot)\ \text{is a V-valued and}\
\mathscr{P}^*\ \text{measurable process defined on}\ [t,T]\ \text{and} \\
&\mathbb{E}\int_t^Tv^2(s)ds<\infty\big\},
\end{split}
\end{equation*}
where $V$ is a compact set of $\mathbb{R}^m$.

We consider the controlled system parameterized by the initial data $(t,x)\in
\mathcal{T}\times\mathbb{R}^n$:
\begin{equation}  \label{sde}
\left\{
\begin{split}
dX^{t,x;v}_s&~=~b(s,X^{t,x;v}_s,v_s)ds+\sigma(s,X^{t,x;v}_s)dW_s+
\pi(s,X^{t,x;v}_s)dB_s, \\
X^{t,x;v}_t&~=~x,
\end{split}
\right.
\end{equation}
where the coefficients
\begin{equation*}
b:\Omega\times\mathcal{T}\times\mathbb{R}^n\times V\rightarrow \mathbb{R} ^n,\
\ \sigma:\Omega\times\mathcal{T}\times\mathbb{R}^n\rightarrow \mathbb{R} ^n,\
\pi:\Omega\times\mathcal{T}\times\mathbb{R}^n\rightarrow \mathbb{R}^n,
\end{equation*}
satisfy

(A1) $b$, $\sigma$ and $\pi$ are bounded functions and for every $(x,v)\in
\mathbb{R}^n\times V$, $b(\cdot,x,v)$, $\sigma(\cdot,x)$ and $\pi(\cdot,x)$ are
$\mathscr{P}$ measurable processes.

(A2) There exists $L>0$ such that
\begin{equation*}
|b(t,x,v)-b(t,x^{\prime },v^{\prime })|+|\sigma(t,x)-\sigma(t,x^{\prime
})|+|\pi(t,x)-\pi(t,x^{\prime })|\leq L(|x-x^{\prime }|+|v-v^{\prime }|).
\end{equation*}

For a given admissible control $v(\cdot)\in\mathcal{V}_{t,T}$, we consider the
following BSDE
\begin{equation}  \label{bsde}
\left\{
\begin{split}
dY^{t,x;v}_s&~=-f(s,X^{t,x;v}_s,Y^{t,x;v}_s,Z^{t,x;v}_s,v_s)ds+\tilde{Z}
^{t,x;v}_sdW_s+\bar{Z}^{t,x;v}_sdB_s, \\
Y^{t,x;v}_T&~=\phi(X^{t,x;v}_T),
\end{split}
\right.
\end{equation}
where we denote $Z=(\tilde{Z},\bar{Z})$. We assume that

(A3) $f:\Omega\times\mathcal{T}\times\mathbb{R}^n\times\mathbb{R}\times
\mathbb{R}^2\times V\rightarrow \mathbb{R}$ satisfies condition (\textbf{H}3).

(A4) The terminal value $\phi:\Omega\times\mathbb{R}^n\rightarrow \mathbb{R}$
is $\mathscr{F}_T\times\mathscr{B}(\mathbb{R}^n)$ measurable and $\phi\in
L^2(\Omega\times\mathbb{R}^n)\cap L^{\infty}(\Omega\times\mathbb{R}^n)$.
\newline
It is well known that under conditions (A1), (A2), (A3) and (A4), SDE
\eqref{sde} and BSDE \eqref{bsde} have unique solutions, respectively.

For a given admissible control $v(\cdot)\in\mathcal{V}_{t,T}$, we introduce the
associated cost functional
\begin{equation*}
J(t,x;v(\cdot))=\mathbb{E}^{\mathscr{F}_t}Y^{t,x;v}_t.
\end{equation*}
Thus the value function of the stochastic optimal control problem is
\begin{equation*}
u(t,x):=\mathop{ess~\inf}_{v(\cdot)\in\mathcal{V}_{t,T}}J(t,x;v(\cdot)).
\end{equation*}
Since the related coefficients $b$, $\sigma$, $f$ and $\phi$ are random
functions, the value function $u$ is a random field. We recall the generalized
dynamic programming principle for the above control problem with recursive cost
functional in \cite{YaPeFaWu97}. For given initial data $
(t,x)\in\mathcal{T}\times\mathbb{R}^n$, a positive number $\delta\leq T-t$, and
a random variable $\eta\in L^2(\Omega,\mathscr{F}^*_{t+\delta},\mathbb{P}
;\mathbb{R})$, we denote a backward semigroup by
\begin{equation*}
G^{t,x;v}_{t,t+\delta}[\eta]:=Y_t.
\end{equation*}
Here $(Y_s,Z_s)_{s\in[t,t+\delta]}$ is the solution to the following BSDE
\begin{equation*}
\left\{
\begin{split}
dY_s&~=~-f(s,X^{t,x;v}_s,Y_s,Z_s,v_s)ds+\tilde{Z}_sdW_s+\bar{Z}_sdB_s,\ s\in[
t,t+\delta], \\
Y_{t+\delta}&~=~\eta,
\end{split}
\right.
\end{equation*}
where $X^{t,x;v}_\cdot$ is the solution to SDE \eqref{sde}.

Then we have the generalized dynamic programming principle (Theorem 6.6 in
Section 2 in \cite{YaPeFaWu97})
\begin{equation*}
u(t,x)=\mathop{ess~\inf}_{v(\cdot)\in\mathcal{V}_{t,T}}\mathbb{E}^{
\mathscr{F}_t}G^{t,x;v}_{t,t+\delta}[u(t+\delta,X^{t,x;v}_{t+\delta})].
\end{equation*}
Suppose $u$ is smooth with respect to $(t,x)$, we can use the It\^{o}
-Wentzell's formula (see, e.g. \cite{Peng92a}) and a similar procedure
in \cite{Peng92a} to obtain that the value function $u$ formally satisfies
BSPDE \eqref{e1}. According to our theoretical results, $u$ is a bounded random
field.

\section{Appendix}

\subsection{Proof of Proposition \protect\ref{c5:2-prp1}}

Suppose ~$\xi$~satisfies the following ODE
\begin{equation*}
\xi(t) = \|\varphi^+\|_{L^\infty(\Omega\times\mathcal{D})} + \int_t^T
(\lambda_1 \xi(s) + \|\lambda_0\|_{\mathbb{L}^{\infty}}) ds.
\end{equation*}
Then for any $t\in\mathcal{T}$, we have
\begin{equation*}
\xi(t) = \frac{\|\lambda_0\|_{\mathbb{L}^{\infty}}}{\lambda_1}
(e^{\lambda_1(T-t)}-1) +
e^{\lambda_1(T-t)}\|\varphi^+\|_{L^\infty(\Omega\times\mathcal{D})}.
\end{equation*}
We will prove ~$u(t,x) \leq \xi(t)$~a.e.~$(\omega,x)$.

Denote~$M_1 := \|u\|_{\mathbb{L}^{\infty}} +
\|\varphi\|_{L^{\infty}(\Omega\times\mathcal{D})}$.~~ Define a function $
\Psi_1$ on~$[-M_1,M_1]$ as follows
\begin{equation*}
\Psi_1(v) = \left\{
\begin{array}{ll}
e^{2\lambda v} - (1+2\lambda v + 2\lambda^2 v^2),~~ & \text{when}~v\in
[0,M_1], \\
0, & \text{when}~v\in [-M_1,0].
\end{array}
\right.
\end{equation*}
By simple computation, we know $\Psi_1$ has the properties: $\forall v\in
[-M_1,M_1]$,
\begin{gather*}
\Psi_1(v) \geq 0,~~\Psi_1^{\prime }(v)\geq 0, \\
\Psi_1(v)=0\Leftrightarrow v\leq0, \\
0\leq v\Psi_1^{\prime }(v)\leq 2(M_1+3)\lambda \Psi_1(v), \\
\lambda \Psi_1^{\prime }- \frac{1}{2}\Psi_1^{\prime \prime }\leq 0.
\end{gather*}
By Lemma \ref{c3:2-lem1}, we have
\begin{eqnarray*}
\begin{split}
&\int_{\mathcal{D}}\Psi_1(u(t,x)-\xi(t))dx - \int_{\mathcal{D}
}\Psi_1(\varphi(x)-\xi(T))dx \\
=&~\int_t^T\!\!\int_{\mathcal{D}}\Psi_1^{\prime }(u(s,x)-\xi(s))\Big\{
(a^{ij}u_{x^j} + \sigma^{ik}q^k)_{x^i}(s,x) + f(s,x,u(s,x),u_x(s,x),q(s,x))
\\
&~- (\lambda_1 \xi(s) + \|\lambda_0\|_{\mathbb{L}^{\infty}})\Big\}dxds -
\frac{1}{2}\int_t^T\!\!\int_{\mathcal{D}} \Psi_1^{\prime \prime 2 }dxds \\
&~-\int_t^T\!\!\int_{\mathcal{D}}\Psi_1^{\prime k}(s,x)dxdW^k_s.
\end{split}
\end{eqnarray*}
According to the integration by parts and Lemma \ref{c5:2-2-002}, we have
\begin{eqnarray*}
\begin{split}
&\int_t^T\!\!\int_{\mathcal{D}}\Big\{\Psi_1^{\prime ij}u_{x^j} +
\sigma^{ik}q^k)_{x^i}(s,x) -\frac{1}{2}\Psi_1^{\prime \prime 2}\Big\} dxds \\
=&~-\int_t^T\!\!\int_{\mathcal{D}}\Psi_1^{\prime \prime }(u(s,x)-\xi(s))
\Big(a^{ij}u_{x^i}u_{x^j} + \sigma^{ik}u_{x^i}q^k + \frac{1}{2} |q|^2\Big)
(s,x)dxds \\
\leq&~-\frac{\mu_{0}}{2} \int_t^T\!\!\int_{\mathcal{D}}\Psi_1^{\prime \prime
}(u(s,x)-\xi(s)) (|u_x|^2 + |q|^2)(s,x)dxds.
\end{split}
\end{eqnarray*}
On the other hand, set~$\widetilde{\lambda}_1 = \lambda_1 \mathop{sgn} (u)$,
then
\begin{equation*}
f(s,x,u,u_x,q) \leq \lambda_0(s,x) + \widetilde{\lambda}_1 u + \lambda \mu_{0}
(|u_x|^2 + |q|^2).
\end{equation*}
Noticing that~$(\widetilde{\lambda}_1-\lambda_1) \xi(s)\leq 0$, we have
\begin{eqnarray*}
\begin{split}
&f(s,x,u(s,x),u_x(s,x),q(s,x)) - (\lambda_1 \xi(s) + \|\lambda_0\|_{\mathbb{L
}^{\infty}}) \\
\leq&~ \lambda_0(s,x) + \widetilde{\lambda}_1 u(s,x) + \lambda \mu_{0} (|u_x|^2
+ |q|^2)(s,x) - (\lambda_1 \xi(s) + \|\lambda_0\|_{\mathbb{L}
^{\infty}}) \\
\leq&~ \widetilde{\lambda}_1 (u(s,x)-\xi(s)) + (\widetilde{\lambda}
_1-\lambda_1) \xi(s) + \lambda \mu_{0} (|u_x|^2 + |q|^2)(s,x) \\
\leq&~ \widetilde{\lambda}_1 (u(s,x)-\xi(s)) + \lambda \mu_{0} (|u_x|^2 +
|q|^2)(s,x).
\end{split}
\end{eqnarray*}
Thus,
\begin{eqnarray}  \label{c5:2-2-003}
\begin{split}
&\int_{\mathcal{D}}\Psi_1(u(t,x)-\xi(t))dx - \int_{\mathcal{D}
}\Psi_1(\varphi(x)-\xi(T))dx \\
\leq&~\int_t^T\!\!\int_{\mathcal{D}}\widetilde{\lambda}_1 \Psi_1^{\prime
}(u(s,x)-\xi(s))(u(s,x)-\xi(s))dxds \\
&~+ \int_t^T\!\!\int_{\mathcal{D}}\mu_{0} \big(\lambda \Psi_1^{\prime }-
\frac{1}{2}\Psi_1^{\prime \prime }\big) (u(s,x)-\xi(s))(|u_x|^2 +
|q|^2)(s,x)dxds \\
&~-\int_t^T\!\!\int_{\mathcal{D}}\Psi_1^{\prime k}(s,x)dxdW^k_s.
\end{split}
\end{eqnarray}
In view of the properties that~$\Psi_1$~ possesses, we have
\begin{eqnarray*}
\begin{split}
0~\leq~&\int_{\mathcal{D}}\Psi_1(u(t,x)-\xi(t))dx \\
~\leq~&\int_t^T\!\!\int_{\mathcal{D}}2(M_1+3)\lambda \lambda_1
\Psi_1(u(s,x)-\xi(s))dxds \\
&~-\int_t^T\!\!\int_{\mathcal{D}}\Psi_1^{\prime k}(s,x)dxdW^k_s,~~~~a.s..
\end{split}
\end{eqnarray*}
Taking expectation on both sides of the above inequality, we get
\begin{align*}
0&\leq \mathbb{E}\int_{\mathcal{D}}\Psi_1(u(t,x)-\xi(t))dx \\
&\leq 2(M_1+3)\lambda \lambda_1\int_t^T\mathbb{E} \bigg[\int_{\mathcal{D}}
\Psi_1(u(s,x)-\xi(s))dx\bigg]ds.
\end{align*}
Gronwall's inequality yields
\begin{equation*}
\mathbb{E}\int_{\mathcal{D}}\Psi_1(u(t,x)-\xi(t))dx = 0,~~~~\forall~t\in
\mathcal{T}.
\end{equation*}
Due to $\Psi_1(v)\geq 0$, it holds that for every $t\in\mathcal{T}$,
\begin{equation*}
\Psi_1(u(t,x)-\xi(t)) = 0,~~~~ a.e.~(\omega,x).
\end{equation*}
The fact that $\Psi_1(v)=0\Leftrightarrow v\leq 0$ implies that for every $
t\in\mathcal{T}$,
\begin{equation*}
u(t,x)\leq \xi(t),~~~~a.e.~(\omega,x).
\end{equation*}
In the same way we can also prove that for every $t\in\mathcal{T}$,
\begin{equation*}
u(t,x) \geq - \frac{\|\lambda_0\|_{\mathbb{L}^{\infty}}}{\lambda_1}
(e^{\lambda_1(T-t)}-1) -
e^{\lambda_1(T-t)}\|\varphi^-\|_{L^\infty(\Omega\times\mathcal{D})},~~
a.e.~(\omega,x).
\end{equation*}
So we obtain~\eqref{c5:2-2-004}.

Next we prove \eqref{c5:2-2-005}. Denote $M_2 = \|u\|_{\mathbb{L}^{\infty}}$
.Define a function $\Psi_2$ on $[-M_2,M_2]$ as
\begin{equation*}
\Psi_2(v) = \left\{
\begin{array}{ll}
\frac{1}{2}\lambda^{-2}[ e^{2\lambda v} - (1 + 2\lambda v) ],~~ & \text{when}
~v\in [0,M_2], \\
\Psi_2(-v), & \text{when}~v\in [-M_2,0].
\end{array}
\right.
\end{equation*}
It is easy to verify that $\Psi_2$ has the following properties: for every~$
v\in [-M_2,M_2]$,
\begin{gather*}
\Psi_2(v)\geq 0,~~\Psi_2^{\prime }(0) = 0,~~ |\Psi_2^{\prime }(v)| \leq
\frac{e^{2\lambda M_2} -1}{\lambda}, \\
\frac{1}{2}\Psi_2^{\prime \prime }(v) - \lambda |\Psi_2^{\prime }(v)| = 1.
\end{gather*}
Applying It\^{o}'s formula to compute $\int_{\mathcal{D}}\Psi_2(u(t,x))dx$, we
have
\begin{equation}  \label{c1}
\begin{split}
&\int_{\mathcal{D}}\Psi_2(u(t,x))dx - \int_{\mathcal{D}}\Psi_2(\varphi(x))dx
\\
\leq~&\int_t^T\!\!\int_{\mathcal{D}} |\Psi_2^{\prime
}(u(s,x))|(\lambda_0(s,x) + \lambda_1|u(s,x)|)dxds \\
&+ \int_t^T\!\!\int_{\mathcal{D}}\mu_{0} \big(\lambda |\Psi_2^{\prime }| -
\frac{1}{2}\Psi_2^{\prime \prime }\big) (u(s,x))(|u_x|^2 + |q|^2)(s,x)dxds \\
&-\int_t^T\!\!\int_{\mathcal{D}}\Psi_2^{\prime k}(s,x)dxdW^k_s.
\end{split}
\end{equation}
Since $\Psi_2$ and $\Psi_2^{\prime }$, defined on the finite duration $
[-M_2,M_2]$, are of the same order as $v^2$ and $v$ near the zero respectively,
there exist positive constants $k_1$, $k_2$, $k_3$ and $k_4$ depending only on
$\lambda$ and $M_2$, such that
\begin{equation*}
k_1v^2\leq\Psi_2(v)\leq k_2v^2,\ \ k_3|v|\leq|\Psi_2^{\prime }(v)|\leq k_4|v|.
\end{equation*}
Thus,
\begin{equation*}
\begin{split}
&\int_t^T\!\!\int_{\mathcal{D}}|\Psi_2^{\prime }(u(s,x))|\lambda_0(s,x)dxds
\\
\leq~&\frac{k_4^2}{2}\int_t^T\!\!\int_{\mathcal{D}}|u(s,x)|^2dxds +\frac{1}{2
}\int_t^T\!\!\int_{\mathcal{D}}\lambda^2_0(s,x)dxds.
\end{split}
\end{equation*}
Taking expectation on both sides of \eqref{c1}, we obtain
\begin{equation}  \label{c2}
\begin{split}
&\mu_{0} \mathbb{E}\int_t^T\!\!\int_{\mathcal{D}}(|u_x|^2 +
|q|^2)(s,x)dxds+k_1\mathbb{E}\int_{\mathcal{D}}|u(t,x)|^2dx \\
\leq~&k_2\mathbb{E}\int_{\mathcal{D}}|\varphi(x)|^2dx+\frac{1}{2}\mathbb{E}
\int_0^T\!\!\int_{\mathcal{D}}\lambda^2_0(s,x)dxds \\
&+(\frac{k_4^2}{2}+k_4\lambda_1)\int_t^T\!\mathbb{E}\!\int_{\mathcal{D}
}|u(s,x)|^2dxds.
\end{split}
\end{equation}
Gronwall's inequality yields
\begin{equation*}
\sup_{t\in\mathcal{T}}\mathbb{E}\int_{\mathcal{D}}|u(t,x)|^2dx\leq\big(\frac{
k_2}{k_1} \|\varphi(x)\|_{L^2(\Omega\times\mathcal{D})}+\frac{1}{2k_1}
\|\lambda_0\|_{\mathbb{L}^{2}}\big) e^{\frac{k_4^2+2k_4\lambda_1}{2k_1}T}.
\end{equation*}
Again from \eqref{c2} we deduce that
\begin{equation*}
\|u_x\|_{\mathbb{H}^0(\mathcal{D})}^2 + \|q\|_{\mathbb{H}^0(\mathcal{D})}^2
\leq C_1,
\end{equation*}
where $C_1$ depends on $\|\varphi(x)\|_{L^2(\Omega\times\mathcal{D})}$, $
\|\lambda_0\|_{\mathbb{L}^{2}}$, $\mu_{0}$, $\lambda$, $\lambda_1$ and $T$. The
proof of Proposition \ref{c5:2-prp1} is complete.

\subsection{Proof of Proposition \protect\ref{c5:prop.01}}

Since the sequence $(u^n)_n$ is monotone and bounded, there exists its limit
function which we denote by $u$. Obviously $u\in\mathbb{L}^\infty$. By the
monotone convergence theorem, $\lim_{n\rightarrow \infty}\|u-u^n\|_{\mathbb{H
}^0(\mathcal{D})}^2 = 0$.

We know from \eqref{c5:2-2-005} in Proposition \ref{c5:2-prp1} that for any $
n\in \mathbb{N}$,

\begin{center}
$ \|u^n\|_{\mathbb{H}^1(\mathcal{D})}^2 + \|q^n\|_{\mathbb{H}^0(\mathcal{D}
)}^2  \leq C_1. $
\end{center}

So we can extract a subsequence $\{n^{\prime }\}$ of the sequence $\{n\}$ and
find functions $v\in \mathbb{H}^1(\mathcal{D};\mathbb{R}^{d})$~and~$q\in
\mathbb{H}^0(\mathcal{D};\mathbb{R}^{d_0})$ such that

\begin{center}
$u^{n^{\prime }}\rightarrow v$ weakly in $\mathbb{H}^1(\mathcal{D})$,\\[0pt]
$q^{n^{\prime }}\rightarrow q$ weakly in $\mathbb{H}^0(\mathcal{D})$.
\end{center}

The uniqueness of limit implies $v=u$.

Next we finish the proof by three steps.\medskip

\emph{Step 1}. Due to the existence of the nonhomogeneous term $f$, the weak
convergence of $(u^{n^{\prime }},q^{n^{\prime }})_{n^{\prime }}$ can not assure
that the limit $(u,q)$ is a weak solution to BSPDE \eqref{c5:main}-
\eqref{c5:con}. Now we prove that the sequences $(u^{n}_x)_{n}$ and $
(q^{n})_{n}$ converge strongly in $\mathbb{H}^0(\mathcal{D})$.

We first deduce from condition (b) that for any $n,m\in\mathbb{N}$,
\begin{eqnarray}  \label{c5:2-2-006}
\begin{split}
&|f^n(t,x,u^n,u^n_x,q^n)-f^m(t,x,u^m,u^m_x,q^m)| \\
\leq&~2 \lambda_2(t,x) +5\lambda\mu_{0}(|u^n_x-u^m_x|^2 + |u^n_x-u_x|^2 +
|u_x|^2 \\
&~+|q^n-q^m|^2 + |q^n-q|^2 + |q|^2).
\end{split}
\end{eqnarray}
Define a function $\Psi_3$ on $[0,2M]$ as follows

\begin{center}
$\Psi_3(v) = \displaystyle\frac{1}{200\lambda^2} (e^{20\lambda v} - 20\lambda v
-1), $
\end{center}

It is easy to verify that $\Psi_3$ is an increasing function and that

\begin{center}
$\Psi_3^{\prime }(0)= \Psi_3(0)=0,~~~~ \displaystyle\frac{1}{2} \Psi_3^{\prime
\prime }-10\lambda\Psi_3^{\prime }\equiv 1.$
\end{center}

For notational simplicity, we denote~$u^{\infty} = u,~q^{\infty} = q$,

\begin{center}
$\delta_u^{n,m} = u^n-u^m,~~~~\delta_q^{n,m} = q^n-q^m.$
\end{center}

By Lemma \ref{c3:2-lem1} and the integration by parts, we have
\begin{align*}
&\int_{\mathcal{D}}\Psi_3(\delta_u^{n,m}(0,x))dx - \int_{\mathcal{D}
}\Psi_3(\delta_u^{n,m}(T,x))dx \\
=& \int_0^T\!\!\int_{\mathcal{D}}\Psi_3^{\prime }(\delta_u^{n,m}(t,x))
[f^n(t,x,u^n,u^n_x,q^n)-f^m(t,x,u^m,u^m_x,q^m)]dxdt \\
&~ -\int_0^T\!\!\int_{\mathcal{D}}\Psi_3^{\prime \prime }(\delta_u^{n,m}(t,x))
[a^{ij}(\delta_u^{n,m})_{x^i}(\delta_u^{n,m})_{x^j}
+\sigma^{ik}(\delta_u^{n,m})_{x^i}(\delta_q^{n,m})^k](t,x)dxdt \\
&~ -\frac{1}{2}\int_0^T\!\!\int_{\mathcal{D}}\Psi_3^{\prime \prime
}(\delta_u^{n,m}(t,x)) |\delta_q^{n,m}(t,x)|^2 dxdt \\
&~ -\int_0^T\!\!\int_{\mathcal{D}}\Psi_3^{\prime }(\delta_u^{n,m}(t,x))
\delta_q^{n,m}(t,x)dxdW_t.
\end{align*}
Noticing $\Psi_3^{\prime }\geq 0$, Lemma~\ref{c5:2-2-002} and
\eqref{c5:2-2-006}, we have
\begin{align*}
&\int_{\mathcal{D}}\Psi_3(\delta_u^{n,m}(0,x))dx - \int_{\mathcal{D}
}\Psi_3(\delta_u^{n,m}(T,x))dx \\
\leq& \int_0^T\!\!\int_{\mathcal{D}}\Psi_3^{\prime }(\delta_u^{n,m})
\times[2\lambda_2 +5\lambda\mu_{0}(|(\delta_u^{n,m})_{x}|^2 +
|\delta_q^{n,m}|^2 \\
&\qquad\qquad + |(\delta_u^{n,\infty})_x|^2 +
|\delta_q^{n,\infty}|^2+|u_x|^2 + |q|^2)](t,x)dxdt \\
&~ -\frac{\mu_{0}}{2}\int_0^T\!\!\int_{\mathcal{D}}\Psi_3^{\prime \prime
}(\delta_u^{n,m}) [|(\delta_u^{n,m})_{x}|^2 + |\delta_q^{n,m}|^2](t,x) dxdt
\\
&~-\int_0^T\!\!\int_{\mathcal{D}}\Psi_3^{\prime }(\delta_u^{n,m}(t,x))
\delta_q^{n,m}(t,x)dxdW_t.
\end{align*}
Taking expectation on both side of the above inequality, we get
\begin{align*}
&\mathbb{E}\int_{\mathcal{D}}\Psi_3(\delta_u^{n,m}(0,x))dx-\mu_{0}\mathbb{E}
\int_0^T\!\!\int_{\mathcal{D}}5\lambda\Psi_3^{\prime }(\delta_u^{n,m})
[|(\delta_u^{n,\infty})_x|^2+ |\delta_q^{n,\infty}|^2](t,x)dxdt \\
&+ \mu_{0} \mathbb{E} \int_0^T\!\!\int_{\mathcal{D}}\Big[\frac{1}{2}
\Psi_3^{\prime \prime }-5\lambda\Psi_3^{\prime }\Big](\delta_u^{n,m})
[|(\delta_u^{n,m})_{x}|^2 + |\delta_q^{n,m}|^2](t,x) dxdt \\
\leq&~ \mathbb{E}\int_{\mathcal{D}}\Psi_3(\delta_u^{n,m}(T,x))dx + \mathbb{E}
\int_0^T\!\!\int_{\mathcal{D}}\Psi_3^{\prime }(\delta_u^{n,m}) [2\lambda_2 +
5\lambda\mu_{0}(|u_x|^2 + |q|^2)](t,x)dxdt.
\end{align*}
Letting $m$ tend to infinity along the subsequence $\{n^{\prime }\}$, together
with the fact that $u^n$ converges pointwise to $u$, we can deduce from the
dominated convergence theorem that
\begin{align*}
&\mathbb{E}\int_{\mathcal{D}}\Psi_3(\delta_u^{n,\infty}(0,x))dx-\mu_{0}
\mathbb{E}\int_0^T\!\!\int_{\mathcal{D}}5\lambda\Psi_3^{\prime
}(\delta_u^{n,\infty}) [|(\delta_u^{n,\infty})_x|^2+
|\delta_q^{n,\infty}|^2](t,x)dxdt \\
&+ \varliminf_{\{n^{\prime }\}\ni m\rightarrow\infty} \mu_{0} \mathbb{E}
\int_0^T\!\!\int_{\mathcal{D}}\Big[\frac{1}{2}\Psi_3^{\prime \prime
}-5\lambda\Psi_3^{\prime }\Big](\delta_u^{n,\infty})
[|(\delta_u^{n,m})_{x}|^2 + |\delta_q^{n,m}|^2](t,x) dxdt \\
\leq &~ \mathbb{E}\int_{\mathcal{D}}\Psi_3(\delta_u^{n,\infty}(T,x))dx +
\mathbb{E}\int_0^T\!\!\int_{\mathcal{D}}\Psi_3^{\prime }(\delta_u^{n,\infty})
[2\lambda_2 + 5\lambda\mu_{0}(|u_x|^2 + |q|^2)](t,x)dxdt.
\end{align*}
In virtue of the weak convergence of the two sequences $(u^{n^{\prime
}}_x)_{n^{\prime }}$~and~$(q^{n^{\prime }})_{n^{\prime }}$, we have
\begin{eqnarray*}
&&\mathbb{E} \int_0^T\!\!\int_{\mathcal{D}}\Psi_3^{\prime
}(\delta_u^{n,\infty}) [|(\delta_u^{n,\infty})_{x}|^2 +
|\delta_q^{n,\infty}|^2](t,x) dxdt \\
&&\leq \varliminf_{\{n^{\prime }\}\ni m\rightarrow\infty} \mathbb{E}
\int_0^T\!\!\int_{\mathcal{D}}\Psi_3^{\prime }(\delta_u^{n,\infty})
[|(\delta_u^{n,m})_{x}|^2 + |\delta_q^{n,m}|^2](t,x) dxdt.
\end{eqnarray*}
Since $\frac{1}{2}\Psi_3^{\prime \prime }-10\lambda\Psi_3^{\prime }\equiv 1$ ,
we have
\begin{align*}
&\mathbb{E}\int_{\mathcal{D}}\Psi_3(\delta_u^{n,\infty}(0,x))dx +
\varliminf_{\{n^{\prime }\}\ni m\rightarrow\infty} \mu_{0} \mathbb{E}
\int_0^T\!\!\int_{\mathcal{D}} [|(\delta_u^{n,m})_{x}|^2 +
|\delta_q^{n,m}|^2](t,x) dxdt \\
& ~~\leq \mathbb{E}\int_{\mathcal{D}}\Psi_3(\delta_u^{n,\infty}(T,x))dx +
\mathbb{E}\int_0^T\!\!\int_{\mathcal{D}}\Psi_3^{\prime }(\delta_u^{n,\infty})
[2\lambda_2 + 5\lambda\mu_{0}(|u_x|^2 + |q|^2)](t,x)dxdt.
\end{align*}
The resonance theorem yields that
\begin{eqnarray*}
\begin{split}
&\mathbb{E}\int_{\mathcal{D}}\Psi_3(\delta_u^{n,\infty}(0,x))dx + \mu_{0}
\mathbb{E} \int_0^T\!\!\int_{\mathcal{D}} [|(\delta_u^{n,\infty})_{x}|^2 +
|\delta_q^{n,\infty}|^2](t,x) dxdt \\
\leq&~ \mathbb{E}\int_{\mathcal{D}}\Psi_3(\delta_u^{n,\infty}(T,x))dx +
\mathbb{E}\int_0^T\!\!\int_{\mathcal{D}}\Psi_3^{\prime }(\delta_u^{n,\infty})
[2\lambda_2 + 5\lambda\mu_{0}(|u_x|^2 + |q|^2)](t,x)dxdt.
\end{split}
\end{eqnarray*}
Noticing again $u^n$ converges pointwise to $u$, $\delta_u^{n,\infty}$
converges to $0$. Therefore, the dominated converge theorem yields
\begin{equation*}
\varlimsup_{n\rightarrow\infty}\mathbb{E} \int_0^T\!\!\int_{\mathcal{D}}
[|(\delta_u^{n,\infty})_{x}|^2 + |\delta_q^{n,\infty}|^2](t,x) dxdt = 0,
\end{equation*}
which implies that $(u^n_x)_n$~and~$(q^n)_n$ converge strongly to ~$u_x$~and~
$q$ in $\mathbb{H}^0(\mathcal{D})$, respectively.\medskip

\emph{Step 2}. We prove that $(u,q)$ is a weak solution to BSPDE
\eqref{c5:main}-\eqref{c5:con}. To this end, we need the following lemma, the
proof of which can be obtained by the same argument as that used in Lemma 2.5
in \cite{Koby00}.

\begin{lem}
Suppose a sequence $(v^n)_n$ converges to $v$ strongly in $\mathbb{H}^0(
\mathcal{D})$. Then there exists a subsequence $(v^{n_k})_k$ such that $
(v^{n_k})_k$ converges to $v$ a.e. and $\widetilde{v}  :=\sup_{k}|v^{n_k}| \in
\mathbb{H}^0(\mathcal{D})$.
\end{lem}

According to the above lemma, we can extract a subsequence $\{n_k\}$ such that

\begin{center}
$u^{n_k}_x\rightarrow u_x,~~ d\mathbb{P}\times dt\times dx\text{-a.e.}$ and
$\widetilde{v}:= \sup_{k}|u^{n_k}_x|\in \mathbb{H}^0(\mathcal{D})$,\\[0pt]
$q^{n_k}\rightarrow q,~~ d\mathbb{P}\times dt\times dx\text{-a.e.}$ and  $
\widetilde{q}:= \sup_{k}|q^{n_k}|\in \mathbb{H}^0(\mathcal{D})$.\\[0pt]
\end{center}

Then it follows from condition (a) that for a.e. $(\omega,t,x)\in\Omega\times
\mathcal{T}\times\mathcal{D}$,
\begin{equation*}
\lim_{k\rightarrow \infty}f^{n_k}(t,x,u^{n_k}(t,x),u^{n_k}_x(t,x),q^{n_k}(t,x))
= f(t,x,u(t,x),u_x(t,x),q(t,x)).
\end{equation*}
On the other hand, we have
\begin{equation*}
|f^{n_k}(t,x,u^{n_k},u^{n_k}_x,q^{n_k})| \leq \lambda_2(t,x) +
\lambda\mu_{0}\sup_k (|u^{n_k}_x|^2 + |q^{n_k}|^2) \leq \lambda_2(t,x) +
\lambda\mu_{0} (|\widetilde{v}|^2 + |\widetilde{q}|^2).
\end{equation*}
The dominated convergence theorem yields
\begin{eqnarray*}
\lim_{k\rightarrow\infty}\mathbb{E}\int_0^T\!\!\int_{\mathcal{D}} \Big|
f^{n_k}(t,x,u^{n_k}(t,x),u^{n_k}_x(t,x),q^{n_k}(t,x)) ~~~~~~ \\
- f(t,x,u(t,x),u_x(t,x),q(t,x))\Big|\,dxdt = 0.
\end{eqnarray*}
In view of the strong convergence of $(u^{n_k}_x)$~and~$(q^{n_k})$ in $
\mathbb{H}^0(\mathcal{D})$, we obtain that $(u,q)$ is a weak solution to BSPDE
\eqref{c5:main}-\eqref{c5:con}. \medskip

\emph{Step 3}. Finally we prove that $u\in L^{2}(\Omega ;C(\mathcal{T};L^{2}(
\mathcal{D})))$. Applying It\^{o}'s formula to $\Vert u^{n_{k}}(t,\cdot
)-u^{n_{l}}(t,\cdot )\Vert _{L^{2}(\mathcal{D})}^{2}$ and proceeding several
standard computation, we get that
\begin{align*}
& \mathbb{E}\sup_{t\in \mathcal{T}}\Vert u^{n_{k}}(t,\cdot
)-u^{n_{l}}(t,\cdot )\Vert _{L^{2}}^{2} \\
& \leq ~\mathbb{E}\int_{0}^{T}\!\!\int_{\mathcal{D}
}|u^{n_{k}}-u^{n_{l}}||f^{n_{k}}-f^{n_{l}}|\,dxds+\mathbb{E}\sup_{t\in
\mathcal{T}}\int_{0}^{T}\!\!\int_{\mathcal{D}
}(u^{n_{k}}-u^{n_{l}})(q^{n_{k}}-q^{n_{l}})\,dxdW_{s} \\
& \leq ~2M\mathbb{E}\int_{0}^{T}\!\!\int_{\mathcal{D}}|f^{n_{k}}-f^{n_{l}}|
\,dxds+\frac{1}{2}\mathbb{E}\sup_{t\in \mathcal{T}}\Vert u^{n_{k}}(t,\cdot
)-u^{n_{l}}(t,\cdot )\Vert _{L^{2}}^{2}+C\Vert q^{n_{k}}-q^{n_{l}}\Vert _{
\mathbb{H}^{0}}^{2}.
\end{align*}
Hence it is easy to see that
\begin{equation*}
\mathbb{E}\sup_{t\in \mathcal{T}}\Vert u^{n_{k}}(t,\cdot )-u^{n_{l}}(t,\cdot
)\Vert _{L^{2}}^{2}\rightarrow 0,~~~~\text{as}~~k,l\rightarrow \infty ,
\end{equation*}
which implies that $\{u^{n_{k}}\}$ is a Cauchy sequence in $L^{2}(\Omega ;C(
\mathcal{T};L^{2}(\mathcal{D})))$, and thus its limit $u\in L^{2}(\Omega ;C(
\mathcal{T};L^{2}(\mathcal{D})))$. The proof of Proposition \ref{c5:prop.01} is
complete.

\bibliographystyle{model1-num-names}

\end{document}